\documentclass[10pt,usenames,dvipsnames]{amsart}
\usepackage{subfigure}
\usepackage{wrapfig}
\usepackage{bm}
\usepackage{pgfplots}

\usepackage{tikz}
\usetikzlibrary{decorations.markings}

\usepackage{xcolor}
\usepackage{verbatim, amssymb,amsmath, amsfonts, amsthm, cases,hyperref,enumerate}
\usepackage[labelformat=empty]{caption}

\newtheorem{theorem}{Theorem}[section]

\newtheorem{proposition}[theorem]{Proposition}
\newtheorem{lemma}[theorem]{Lemma}
\newtheorem{corollary}[theorem]{Corollary}
\newtheorem{remark}[theorem]{Remark}

\newcommand{\sezione}[1]{\section{#1}\setcounter{equation}{0}}

\newcommand{\nor}{\Arrowvert}

\def\R{{\rm I\mskip -3.5mu R}}
\def\N{{\rm I\mskip -3.5mu N}}

\def\cal{\mathcal}
\def\un{u_n}
\def\intO{\int_{\Om}}

\def\ln{\l_n}
\def\S{{\mathcal{S}}}
\def\tu{\tilde u}

\newcommand{\red}{\color{red}}

\def\e{{\varepsilon}}
\def\di12{\mathcal{D}^{1,2}(\R^n)}

\def\d{\delta}

\def\D{\Delta}
\def\l{{\lambda}}

\def\a{{\alpha}}

\def\0l{_{0,\l}}
\def\1l{_{1,\l}}
\def\2l{_{2,\l}}
\def\3l{_{3,\l}}
\def\4l{_{4,\l}}

\def\G{\Gamma}

\def\de{\partial}

\def\Om{\Omega}

\definecolor{purple}{rgb}{0.5, 0.0, 0.5}

\def\sideremark#1{\ifvmode\leavevmode\fi\vadjust{\vbox to0pt{\vss
 \hbox to 0pt{\hskip\hsize\hskip1em
 \vbox{\hsize2.1cm\tiny\raggedright\pretolerance10000
  \noindent #1\hfill}\hss}\vbox to15pt{\vfil}\vss}}}%

\begin{document}
\title[On the critical points ..]{On the critical points of solutions of PDE in a non-convex settings: the case of concentrating solutions}

\author[]{F. Gladiali, M. Grossi}

\address{Francesca Gladiali, Dipartimento SCFMN, Universit\`a  di Sassari  - Via Piandanna 4, 07100 Sassari, Italy}
\address{Massimo Grossi, Dipartimento di Scienze di Base Applicate per l’Ingegneria, Universit\`a degli Studi di Roma \emph{La Sapienza}, P.le Aldo Moro 5, 00185  Roma, Italy}

\thanks{2010 \textit{Mathematics Subject classification:35B09,35B40,35Q}  }

\thanks{\textit{Keywords}: Critical points, multipeaks solutions,index}

\thanks{}

\begin{abstract}
In this paper we are concerned with the number of critical points of solutions of nonlinear elliptic equations. We will deal with the case of non-convex, contractile and non-contractile planar domains.
We will prove results on the estimate of their number as well as their index. In some cases we will provide the exact calculation. The toy problem concerns the multi-peak solutions of the Gel'fand problem, namely
\begin{equation*}
            \left\{\begin{array}{lc}
                        -\Delta u=\l e^{u}  &
            \mbox{  in }\Om\\
              u=0 & \mbox{ on }\partial \Om,
                      \end{array}
                \right.
\end{equation*}
where $\Om\subset \R^2$ is a bounded smooth domain and $\l>0$ is a small parameter.
\end{abstract}

\maketitle

\section{Introduction}\label{s1}

The calculation of the number of critical points of solutions of nonlinear elliptic differential equations is an old and fascinating problem. Many questions are still largely unresolved and we are far from a complete description of the problem.
In this paper we will limit ourselves to considering $positive$ solutions of the following problem,
\begin{equation}\label{0}
\begin{cases}
-\Delta u=f(u)&in\ \Om\\
u=0&on\ \partial\Om,
\end{cases}
\end{equation}
where $\Om$ is a bounded smooth domain of $\R^N$, $N\ge2$ and $f$ is a Lipschitz nonlinearity.
In this setting the first result to mention is certainly that by Gidas, Ni and Nirenberg (see \cite{gnn}) where the authors prove, for convex and symmetric domains, the uniqueness of the critical point of positive solutions.

An important line of research was that of removing the $symmetry$ assumption and proving the  Gidas, Ni and Nirenberg theorem only assuming the convexity of $\Om$. The answer to this question is thought to be affirmative, indeed no counter-example has been provided.
However this seems to be a very difficult open problem. Indeed, if $\Om$ is a convex (not necessarily symmetric) domain, the uniqueness of the critical point of the solution has been proved as consequence of the strict convexity of the level sets of solutions  with appropriate nonlinearities  (see \cite{ml}, \cite{bl}, \cite{app}, \cite{k}, \cite{cf2}, \cite{gm}).
Observe that each functions  $f$ appearing in the mentioned results is handled differently, actually it is not possible to prove the strict convexity of the level sets of $u$ for generic $f$ (see the example in \cite{hns}).
A result for general nonlinearities $f$ for planar domain with positive curvature was proved by Cabrè and Chanillo, but here the solutions are required to be semi-stable, (see \cite{dgm} for an extension to $\Om$ with non-negative curvature). Unfortunately,
this assumption is not verified by mountain pass solutions as well as by many nonlinearities in \eqref{0}
like $f(s)=e^s$ or $f(u)=u^p$ with $p>1$.

The aim of this paper is to deduce some information on the critical points, like the exact number and their index, in planar domains which are not necessarily convex or simply connected.
Given the difficulty of the topic, we decided to choose a model problem in the plane and to consider the class of solutions concentrating at a finite number of points.

The class of solutions we consider are those that appear in Gelfand's problem, although we believe that these techniques can be adapted to deal with other cases, (such as for example with $f(u)=u^p$ with $p$ large).\\
The Gel'fand problem deals with  solutions to the problem,
\begin{equation}\label{1}
            \left\{\begin{array}{lc}
                        -\Delta u=\l e^{u}  &
            \mbox{  in }\Om\\
              u=0 & \mbox{ on }\partial \Om,
                      \end{array}
                \right.
\end{equation}
where $\Om\subset \R^2$ is a bounded smooth domain and $\l>0$ is a parameter.  This problem is associated
with several phenomena in differential geometry,  turbulence
theory,  statistical mechanics and gauge field theory (see the paper of Gel'fand,  \cite{ge9} or \cite{nasu}, \cite{su2} and the
references therein). Problem \eqref{1} has been derived in the context of statistical mechanics in \cite{clmp1} while the ties with the turbolence theory of an incompressible and homogeneus fluid, via the Euler equation is described in the introduction of \cite{dem}. 
This problem can be regarded as a simplified model to describe the steady states of reaction-diffusion equations where the diffusion is of exponential type.  
Reaction-diffusion equations arise in a wide variety of biological, physiological and chemical contests
such as the spreading of a chemical substance or the propagation of a disease, and in the study of complex systems.  One can see, as an example the books of Murray \cite{mu1, mu2} where spatial models for biomedical applications are considered.
In this contest the Laplace operator effectively describes the spatial variation of physiological quantities within a tissue or other (chemical) quantities of the model and the parameter $\lambda$ represents a scaling factor or an intensity parameter that can be related to the diffusion coefficient.  
The interpretation of the solution properties to stationary problems like \eqref{1}, in the context of physiological models,  requires careful evaluation of the involved processes and the integration with other clinical and experimental information.
Nevertheless properties of solutions to \eqref{1} can be of great interest for a deeper understanding of the system, its dynamics, and emergent phenomena. 
The existence and localization of critical points of solutions to reaction-diffusion equations can provide information about the concentration, diffusion, and interaction of substances that play a role.  They can be used to study wave propagation or the formation of spatial patterns.
It is a fundamental tool for the theoretical study and mathematical modeling of reaction and diffusion processes in various scientific and engineering contexts.

Finally, it is interesting to observe the connection between \eqref{1} with the Mean Field equation in a bounded domain of $\R^2$, namely
\begin{equation}\label{MF}
            \left\{\begin{array}{lc}
                        -\Delta v=\rho \frac{ h(x)e^{v}}{\int_{\Omega} h(x)e^{v}dx}   &
            \mbox{  in }\Om\\
              v=0 & \mbox{ on }\partial \Om.
                      \end{array}
                \right.
\end{equation}
By \cite{cli}, Section 6, the Mean Field equation \eqref{MF} with $h(x)=1$ in $\Omega$ is equivalent to \eqref{1} just letting $\lambda=\frac {\rho}{ \int_{\Omega} e^{v}dx} $.

\

Let us start by recalling some known facts about the structure of the solutions to  \eqref{1} (see {\cite{cr2}} and also \cite{cli} for a more detailed construction of the solutions). 

Let ${\cal X}=\{
(\lambda,u)\in\R^+\times C(\Omega) \mid (\ref{1})\hbox{ is
satisfied}\}$ be the solutions set to \eqref{1}. The first observation is that
${\cal X}=\emptyset$ for $\lambda$ large enough.  In particular there exists $\lambda^*=\lambda^*(\Omega)$ such that \eqref{1} admits a unique stable solution $\underline{u}_\lambda$ for every $\lambda \in [0,\lambda^*)$, called $minimal$ solution. Corresponding to $\lambda^*$ there exists a unique solution to \eqref{1}, and the solution curve bends back, so that there exists at least $two$ solutions of \eqref{1} when $\lambda\in (\lambda^*-\delta, \lambda^*)$ if $\delta $ is small enough.  This curve of solutions is sufficiently smooth for every domain $\Omega$. Further for every $\lambda\in (0,\lambda^*)$ there exists at least another non stable solution $u_\lambda$. 
Critical phenomena in fact occur to these (non stable) solutions
$u=u_\l(x)$ to \eqref{1} as $\lambda\downarrow 0$. The first observation is that $\|u_\lambda\|_{\infty}\to \infty$ as $\lambda\to 0$ by \cite{bremer}. 
This profile is described by
\cite{nasu} as a quantized blow-up mechanism, and we recall it in details in Section \ref{se:2}, see \eqref{2}.
Moreover sequences of blowing-up solutions $u_{\lambda_n}$ can be construct when $\lambda_n\to 0$ according to the topological
and geometrical properties of the domain $\Omega$. In particular in \cite{we} the author constuct a sequence of solutions on simply connected domains $\Omega$ blowing-up at a critical point $x_0\in \Omega$ of the {\em Robin} function $\mathcal R(x)$ of $\Omega$,  see also \cite{mo}.  Non simply connected domains are considered in \cite{djlw}.
 
A first complete existence result, for multipeak solutions, was proved in \cite{bapa} where the authors construct a sequence of solutions that blows-up at $m$ points $\{P_1,\dots, P_m\}$ of $\Omega$ when the point $(P_1, \dots, P_m)$ is a nondegenerate critical point for the {\em Kirchhoff-Routh} function of $\Omega$, that we denote  $\mathcal{KR}_\Omega(x_1,\dots,x_m)$ (see Section \ref{se:2} for its definition).

Observe that these solutions are $never$ semi-stable, so even if $\Om$ is strictly convex, the Cabrè-Chanillo result is not applicable.

 In \cite{egp} and \cite{dkm} this condition was relaxed  requiring that $(P_1, \dots, P_m)$ is a $stable$ critical point of $\mathcal{KR}_\Omega $.\\

Moreover the non-degeneracy of
$(P_1, \dots, P_m)$ as critical point of $\mathcal{KR}_\Omega$ implies the nondegeneracy of $u_\lambda$ for
$\lambda$ small enough.  This was
proven first for $m=1$ by \cite{gg1} and then by \cite{grohsu} in the general case.  

Let $B_\rho(Q)\subset\R^2$ be the ball centered at $Q$ and radius $\rho$.
Our first result is the following,
\begin{theorem}\label{prop-general}
Let $\Omega$ be a smooth domain with $k\geq 0$ holes and let  $u_\l$ be a family of solutions to \eqref{1} that blow-up at $m\geq 1$ points $\{P_1,\dots,P_m\}$ as $\l\to 0$.   Then, when $\l$ is small enough
\begin{equation}\label{nc}
\sharp\{\hbox{critical point of $u_\l$ in $\Om$}\}\ge2m+k-1.
\end{equation}
More precisely we have that, for $\l$ small, there exists exactly one critical point (a non-generate maximum) for $u_\l$ in $B_\rho(P_i)$ $i=1,..,m$ and $\rho$ small. Next, denoting by $D=\Om\setminus\cup_{i=1}^m B_\rho(P_i)$ and $\mathcal C_\l$ the set of critical points of $u_\l$ in $D$ we have that $u_\l$ admit at least $m+k-1$ nondegenerate saddle points in $D$ and
\begin{equation}\label{nb}
\sum_{z_j\in \mathcal C_\l} index_{z_j} (\nabla u_\l)=1-k-m.
\end{equation}
Moreover, for any $k,m\ge1$, there exists $\tilde\Om$ and a corresponding family of solutions $u_\l$ that blow-up at $m\geq 1$ points $\{P_1,\dots,P_m\}$ such that, again for $\l$ small enough,
\begin{equation}\label{nc2}
\sharp\{\hbox{critical point of $u_\l$ in $\tilde\Om$}\}=2m+k-1.
\end{equation}
Moreover all critical points of $u_\l$ are nondegenerate, $m$ of them are local maxima and  $m+k-1$ saddle points.
\end{theorem}
\vskip0.2cm
\begin{center}
\boxed{
\begin{tikzpicture}[scale=0.8]
  \draw[blue] (0,0) circle (1cm); 
  \draw[white, line width=3cm] (2, -0.15) arc (0:16:2cm);
  \draw[blue]  (4,0) circle (1cm); 
  \draw[white, line width=3cm] (3, -0.15) arc (-4:1:3.1cm);
  \draw[blue]  (8,0) circle (1cm); 
  \draw[white, line width=3cm] (5, -0.15) arc (0:8:2cm);
   \draw[white, line width=3cm](7, -0.15) arc (0:8:2cm);
  \draw[blue]  (1,-0.15) -- (3,-0.15); 
  \draw[blue]  (1,0.15) -- (3,0.15); 
  \draw[blue]  (5,-0.15) -- (7,-0.15); 
  \draw[blue]  (5,0.15) -- (7,0.15); 
   \fill (0,0) circle (1pt) node[anchor=north] {$P_1$};
     \fill (4,0) circle (1pt) node[anchor=north] {$P_2$};
       \fill (8,0) circle (1pt) node[anchor=north] {$P_3$};
        \draw[red, fill=red] (2,0) circle (1pt) node[anchor=north] {$z_1$};
                 \draw[red, fill=red] (6,0) circle (1pt) node[anchor=north] {$z_2$};
       \draw (-0.37,0.1) parabola bend (0,2)
(0.3, 0);
\draw (3.7,-0.1) parabola bend (4,2)
(4.3,-0.1);
\draw (7.7,0) parabola bend (8,2)
(8.3,0);
  \fill (14.1,1.6)  node[anchor=north] {Example of a domain with $k=0$ whose solution };
   \fill (14,0.8)  node[anchor=north] {$u_\l$ concentrates at $P_1,P_2,P_3$. We have that $u_\l$};
     \fill (13.8,0)  node[anchor=north] { has $3$ maxima and $2$ saddle points {\red$z_1$} and {\red$z_2$}};
\end{tikzpicture}}
\end{center}
\begin{center}
\boxed{
\begin{tikzpicture}[scale=0.8]
\draw[red] (-1.97,-0.03) arc
    [start angle=0,
        end angle=350,
        x radius=0.4cm,
        y radius =0.4cm] ;

\draw[red] (-1.65,-0.15) arc
    [start angle=0,
        end angle=315,
        x radius=0.7cm,
        y radius =0.6cm] ;

        \draw[blue] (-1.6,-0.15) arc
   [start angle=0,
        end angle=330,
        x radius=1cm,
        y radius =0.9cm] ;
        
                \draw[blue] (-1.45,-0.15) arc
   [start angle=0,
        end angle=320,
        x radius=1.2cm,
        y radius =1cm] ;
\begin{scope}
\clip [] (0,0) circle [radius=1.9cm-0.5\pgflinewidth]; 
\fill[white]  (0,0) circle [radius=2cm];
\end{scope}

\draw[blue] (0,0.05) arc
    [start angle=10,
        end angle=348,
        x radius=1cm,
        y radius =1cm] ;
        
      \draw[blue] (4,0.05) arc
    [start angle=9,
        end angle=171,
        x radius=1cm,
        y radius =1cm] ;  
        
         \draw[blue] (4,-0.35) arc
    [start angle=-13,
        end angle=-168,
        x radius=1cm,
        y radius =1cm] ;    
        
        \draw[blue] (6,0.05) arc
    [start angle=170,
        end angle=-168,
        x radius=1cm,
        y radius =1cm] ;
        
  \draw[blue]  (0,0.042) -- (2.02,0.042); 
    \draw[blue]  (0,-0.342) -- (2.02,-0.342); 
    \draw[blue]  (4,0.042) -- (6.02,0.042); 
    \draw[blue]  (4,-0.342) -- (6.02,-0.342); 
    
       \draw[->] (-4.5,-0.15) -- (9,-0.15) node[right] {$x$}; 
   \draw[->] (3,-1.3) -- (3,2.5) node[right] {$y$};
   \fill (-1,-0.15) circle (1pt) node[anchor=north] {$P_1$};
     \fill (3,-0.15) circle (1pt) node[anchor=north] {$P_2$};
       \fill (7,-0.15) circle (1pt) node[anchor=north] {$P_3$};
            \draw[red, fill=red] (1,-0.15) circle (1pt) node[anchor=north] {$z_2$};
            \draw[red, fill=red] (5,-0.15) circle (1pt) node[anchor=north] {$z_1$};
     \draw[red, fill=red] (-2.9,-0.15) circle (1pt) node[anchor=north] {$z_3$};
      \draw[red, fill=red] (-3.7,-0.15) circle (1pt) node[anchor=north] {$z_4$};
                 
                    \fill (-3,0.65)  node[anchor=north] {\large$\mathcal{C}_1$};
                      \fill (-4,0.65)  node[anchor=north] {\large$\mathcal{C}_2$};
        \draw (-1.27,0) parabola bend (-1.1,2) (-0.67, 0);
\draw (2.7,-0.1) parabola bend (3,2)(3.3,-0.1);
\draw (6.7,0) parabola bend (7,2) (7.3,0);
                      \fill (2,-2)  node[anchor=north] {Example of a domain with $k=2$ whose solution $u_\l$ concentrates at $P_1,P_2,P_3$. We};
                      \fill (-0.55,-2.5)  node[anchor=north] {have that $u_\l$ has $3$ maxima and $4$ saddle points {\red$z_1$},.., {\red$z_4$}};
 \end{tikzpicture} }
 \end{center}
 \vskip0.2cm
 The proof of the previous theorem uses two basic tools,
 \begin{itemize}
 \item Some delicate estimates on the asymptotic behavior of the solution $u_\l$
 \item Techniques of differential topology as the Poincarè-Hopf Theorem and degree theory
 \end{itemize}
Observe that the estimates on the asymptotic behavior of  $u_\l$ use quite heavily the shape of the nonlinearity but it is reasonable to conjecture that can be obtained also for other nonlinearities (for example of power type or for Moser-Trudinger problems).
The techniques of differential topology mainly concerns the computation of critical points of $harmonic$ functions and are independent of the exponential nonlinearity.
 
A natural question which arises from the previous theorem is the following,
\vskip0.2cm
{\em Question $1$.} Let us consider a domain $\Om$ with $k\ge0$ holes. If we consider a solution  to \eqref{1} blowing-up at $m$ points in $\Om$, for what values $m$ and $k$ does equality in \eqref{nc} hold?
\vskip0.2cm
The rest of the paper is devoted to answer to this question. The first result is,
\begin{theorem}\label{T1}
Assume $\Omega\subset\R^2$ and $u_\l$ is a family of solutions to \eqref{1} that blow-up at $m\geq 1$ points $\{P_1,\dots,P_m\}$ as $\l\to 0$. Then we have that
\begin{itemize}
\item[\bf a)] If $m=1$ and $k=0$ (i.e. $\Om$ is simply connected) then any solution $u_\l$ to \eqref{1} has, for $\l$ small,  only $1$ critical point in $\Omega$, (obviously its maximum) which is nondegenerate.
 \item[\bf b)] If $m=2$ and $k=0$ (i.e. $\Om$ is simply connected) then any solution $u_\l$ to \eqref{1} has, for $\l$ small,  $3$ nondegenerate critical points: the maximum points in the balls $B_\rho(P_1)$, $B_\rho(P_2)$  and a saddle point.
\item[\bf c)] If $m=1$ and $k=1$ (i.e $\Om$ has one hole) then any solution $u_\l$ to \eqref{1} has, for $\l$ small, $two$  nondegenerate critical points in $\Omega$; one is the absolute maximum, the second one is a saddle point. 
 \end{itemize}
\end{theorem}
\begin{figure}[h]
\centering
\subfigure[Example of case {\bf a)} where the solution $u_\l$ has $1$ critical point.]
{\includegraphics[width=5cm,height=2.5cm]{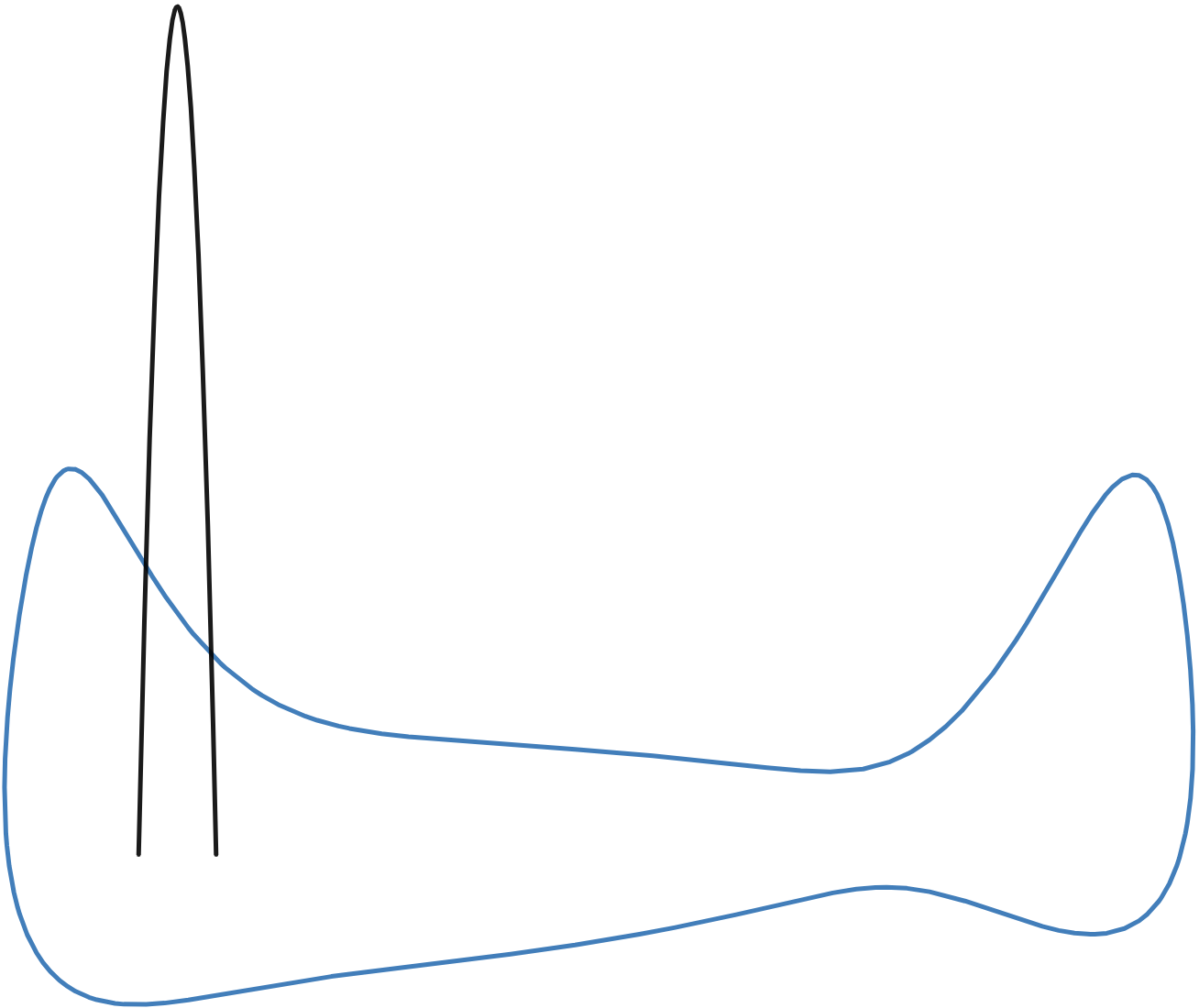}}
\hspace{10mm}
\subfigure[Example of case {\bf c)} where the solution $u_\l$ has $2$ critical points.]
{\includegraphics[width=5cm, height=2.5cm]{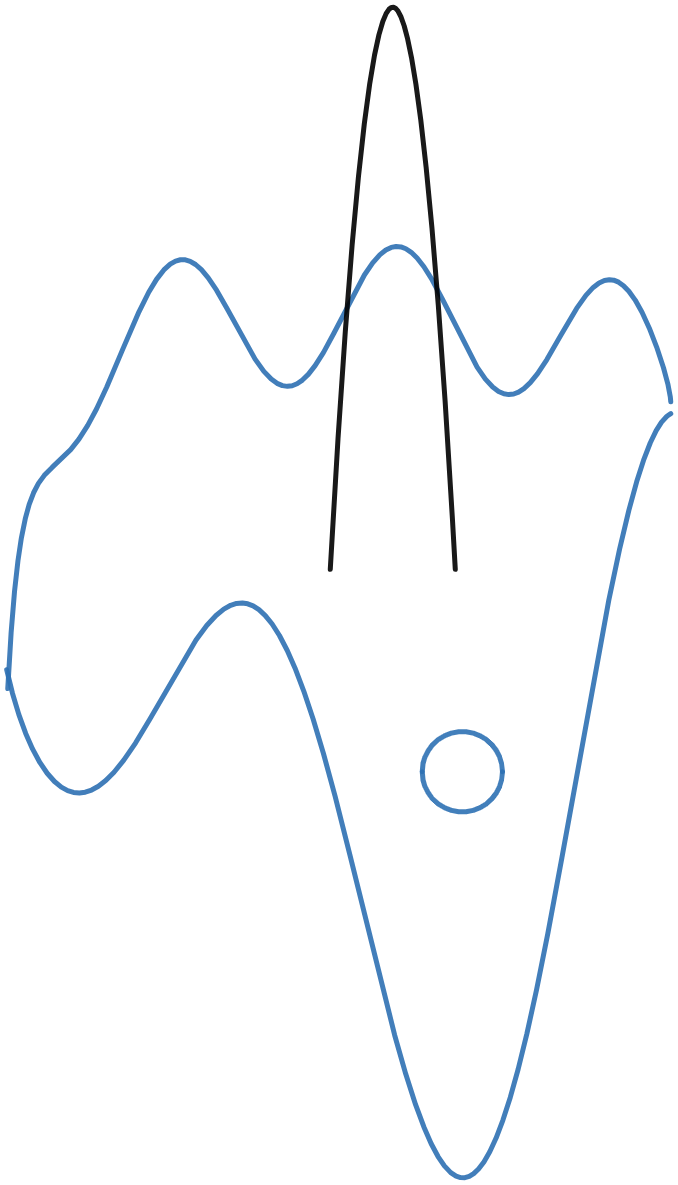}}
\end{figure}
\begin{remark}
The claim $\bf a)$ in Theorem \ref{T1} was proved in \cite{gg1} under the additional assumption of the $convexity$ of $\Om$ (see also \cite{gm} for similar results in higher dimensions). Observe that in \cite{gg1}  it was obtained the stronger claim that the super-level sets of $u_\l$ are star-shaped but this properties will not be true in our setting of general simply-connected domains.\\
\end{remark}
We stress that Theorem \ref{T1} claims that, under the assumption $\bf a)$, $\bf b)$ or $\bf c)$  $any$ blowing-up solution admits the same number of critical points.\\
This is not true outside of this setting as showed in next Theorem.
\begin{theorem}\label{ex}
For every $m\geq 3$ there exists a simply connected domain $\widehat\Om$ and a corresponding family of solutions $u_\l$ to \eqref{1} that blow-up at $m$ points $\{P_1,\dots,P_m\}$ as $\l\to 0$,  such that $u_\l$ has at least $2m+1$ nondegeneratecritical points,  $m+1$ of them are local maxima and $m$ are saddle points. \\
Moreover for every $m\geq 3$ and for every $\lambda>0$ small, there exists a domain $\widehat\Om_\l$, which has $k\geq 1$ holes and such that problem \eqref{1} admits a solution $\widehat u_\l$ in $\widehat\Om_\l$ with at least $2m+k+1$ nondegenerate critical points, $m+1$ of them are local maxima and $m+k$ are saddle points. 
As $\l\to 0$ along a sequence  $ \widehat\Om_\l\to \widehat\Omega$ and  $\widehat u_\l$  blow-up at $m$ points.
\end{theorem}
\begin{remark}
In the proof of Theorem \ref{ex} we also construct another  domain $\Omega$ with $k = mh$ holes, for some $h\in\N$,  such that problem \eqref{1} admits a family of solution $u_\l$ in $\Om$ that blow-up at $m\ge 1$ points $ \{P_1,...,P_m\}$ as $\l\to 0$,  with $2m+k+1$ critical points.  In this alternative case the proof is considerably simpler.
\end{remark}
\begin{center}
\boxed{
\begin{tikzpicture}[scale=0.5]

\draw[blue] (0,0) arc
    [start angle=-20,
        end angle=320,
        x radius=1cm,
        y radius =1cm] ;
        
            \draw[blue] (6,0) arc
    [start angle=200,
        end angle=-140,
        x radius=1cm,
        y radius =1cm] ;

           \draw[blue] (3,-5) arc
    [start angle=100,
        end angle=440,
        x radius=1cm,
        y radius =1cm] ;

 \draw[blue]  (0,0) -- (3.1,-1.7); 
  \draw[blue]  (-0.16,-0.3) -- (2.95,-2.1); 
    \draw[blue]  (6,-0.025) -- (3.1,-1.7); 
   \draw[blue]  (6.2,-0.3) --(3.3,-2.1) ;  
      \draw[blue]  (3.3,-2.1) -- (3.35,-5);  
      \draw[blue]  (2.95,-2.1) -- (3,-5);  

   \fill (-1,0.5) circle (1pt) node[anchor=north] {$P_1$};
     \fill (7,0.5) circle (1pt) node[anchor=north] {$P_2$};
        \fill (3.1,-1.9) circle (2pt) node[anchor=north]{};
      \fill (3.4,-1.9)  node[anchor=north] {$O$};
      \fill (3.2,-5.8) circle (1pt) node[anchor=north] {$P_3$};
           \draw[red, fill=red] (1.3,-0.95) circle (1pt) node[anchor=north] {$z_1$};
                 \draw[red, fill=red] (4.8,-0.95) circle (1pt) node[anchor=north] {$z_2$};
                  \draw[red, fill=red] (3.15,-3.05) circle (1pt) node[anchor=north] {$z_3$};

   \fill (16.9,1.6)  node[anchor=north] {Example of a domain with $k=0$ whose solution };
   \fill (16.3,0.6)  node[anchor=north] {$u_\l$ blows-up at $P_1$, $P_2$, $P_3$. We have that $u_\l$};
     \fill (16.95,-0.6)  node[anchor=north] { has $4$ maxima (note the additional maximum at };                
                  \fill (14.55,-1.6)  node[anchor=north] {$0$) and $3$ saddle points {\red$z_1$}, {\red$z_2$}, {\red$z_3$}};

                     \draw (-1.3,0.5) parabola bend (-1,2.5)(-0.8, 0.5);
\draw (6.7,0.5) parabola bend (7,2.5)(7.3, 0.5);
\draw (2.9,-5.7) parabola bend (3.2,-3.7)(3.5, -5.7);
\draw (2.9,-2) parabola bend (3.1,-1.6)(3.3, -2);
   \end{tikzpicture} }
\end{center}
\vskip0.2cm
\begin{remark}
Theorems \ref{T1} and \ref{ex} provide an ``almost complete'' answer to Question $1$. The cases left open are $m=1$ with $k\ge2$ and $m=2$ with $k\ge1$, where it is not clear if the equality in \eqref{nc} always holds.
\end{remark}
Another natural question arising by the previous result is the following,
\vskip0.2cm
{\em Question $2$.} Let us fix a domain $\Om$ with $k$ holes and $u_\l$ be a solution to \eqref{1} with $m$ peaks. Is there an $upper$ bound on the number of critical point of $u_\l$ depending only by $m$ and $k$?
\vskip0.2cm
We end this introduction with some words on the techniques used in the proof of our results. The first tool we need is a good $C^1$ estimate of the solution $u_\l$ as $\l\to 0$. These results are well known in the literature in the blow-up ball $B_{\delta_{i,\l}\bar R}(P_i)$, for any $\bar R>0$ ,and in the set $\Om\setminus\cup _{i=1}^m B_{\rho}(P_i)$ for any $\rho$ small enough(see Section \ref{se:2}). However no sharp estimate is available in the ``annular'' region $\{\delta_{i,\l}\bar R\le|x-P_i|\le \rho\}$. The knowledge of $u_\l$ in this region is crucial for our computations and it requires delicate estimates (see Section \ref{S2}).

A second tool is given by refined topological arguments on the index of the critical points of harmonic functions. Here we used both classical arguments like the Poincaré-Hopf theorem and some results by Alessandrini and Magnanini in \cite{am}(see Section \ref{ss2}). Resuming, the computation of the critical point of $u_\l$ requires to split $\Om$ in three subdomains,
\begin{itemize}
\item $B_{\delta_{i,\l}\bar R}(P_i)$, for some $\bar R>0$ where there is {\bf one} nondegenerate critical point (of course the maximum)
\item $B_{\rho}(P_i)\setminus B_{\delta_{i,\l}\bar R}(P_i)$ where there is {\bf no} critical point
\item $\Om\setminus\cup_{i=1}^m B_{\rho}(P_i)$ where the number of critical points can vary according with $m$ and $k$.
\end{itemize}
\begin{remark}
As recalled in the beginning, solutions to \eqref{1} are linked to solutions to \eqref{MF}. Then the results in Theorems \ref{prop-general}, \ref{T1}, \ref{ex} hold also for solutions $v_\rho$ to \eqref{MF} that concentrate at $m$ points of $\Omega$ when $\rho-8\pi m$ is small enough.
\end{remark}
\vskip0.2cm
The paper is organized as follows: in Section \ref{se:2} we recall some known facts about the asymptotic behavior of the solution as $\l\to0$ as well as some classical result about the critical point theory. In Section \ref{S2} we refine some asymtotics of the solution $u_\l$ in an annular set  that shrinks slowly.  In Section \ref{S3} we apply the previous results to show that, for every $i=1,\dots,m$,  in $B_\rho(P_i)$ there is only one nondegnerate critical point. In Section  \ref{S4} we analyze the structure of the critical points of $u_\l$ in the remaining part of $\Om$ and prove Theorem \ref{prop-general}. In Section \ref{S5} we prove Theorems \ref{T1} and \ref{ex}.

\section{Preliminaries}\label{se:2}
\vskip0.3cm
\subsection{Asymptotic estimates of the solution $\mathbf{u_\l}$}\label{ss1}
In this section we recall some known fact on the asymptotic behavior of solutions to \eqref{1} as $\l\to0$. Most of the proofs can be found in  \cite{nasu} (see also \cite{mw2}, \cite{gg1} and the references therein).

Let $\{\ln\}_{n\in\N}$ be a sequence of positive values such that $\ln\to 0$ as $n\to \infty$ and let $\un=\un(x)$ be a sequence of solutions of \eqref{1} for $\l=\ln$.

We have the following quantized blow-up mechanism 

\begin{equation}\label{2}
\ln \intO e^{\un}\, dx\rightarrow 8\pi m
\end{equation}
for some $m=0,1,2, \cdots, +\infty$ along a sub-sequence.\\
\begin{itemize}
\item If $m=0$ the pair $({\lambda_n},u_{n} )\in{\cal X}$
converges to
$(0,0)$ as ${\lambda_n}\rightarrow0$ and $u_n$ is the minimal solution with Morse index $0$.\\
\item If $m=+\infty$ there arises the entire blow-up of the solution
$u_n$, in the sense that
$\inf_C u_n\rightarrow+\infty$ for any $C$ compact, $C\subset \Omega$. \\
\item If $0<m<\infty$ the solutions $\{ u_{n}\}$ blow-up at $m$-points.
\vskip0.2cm\noindent
\item Thus there is a set 
\begin{equation}\label{2a}
{\cal S}=\{ P_1, \cdots,P_m\}\subset \Omega
\end{equation}
of $m$ distinct points such
that  $\nor \un \nor_{L^{\infty}(\omega)}=O(1)$
for any $\omega$ compact $\omega\subset \overline{ \Omega}\setminus \S$,\\
\item $\un{|_{\S}}\rightarrow +\infty \quad \hbox{ as }n\to \infty.$
\end{itemize}
Next we describe the pointwise limit of $u_n$. Here and henceforth, $G(x,y)$ denotes the Green function of
$-\Delta$ in $\Om$ with Dirichlet boundary condition,  that is
\begin{equation}
\label{4}
G(x,y)=\frac 1{2\pi}\log{|x-y|^{-1}}+H(x,y)
\end{equation}
where $H(x,y)$ is the regular part of $G(x,y)$ and we denote by $\mathcal R(x)=H(x,x)$ the {\em Robin-function} of $\Omega$. Further we recall the definition the so called {\em Kirchhoff-Routh} function of $\Omega$
$$\mathcal{KR}_\Omega(x_1,\dots, x_m)=\frac 12 \sum_{j=1}^m \mathcal R(x_j)+\frac 12 \sum_{\substack   {1\leq j,h\leq m\\ j\neq h}}G(x_j,x_h).$$ The following results can be found in \cite{mw2} and \cite{su3}.
\begin{theorem}\label{T0}
Let $u_n$
be  a sequence of solutions to \eqref{1} which blows-up at $\{P_1,\dots,P_m\}$ as $ n\to \infty$.
Then we have that
\begin{equation}\label{3}
u_n \rightarrow 8\pi \sum_{j=1}^m \,G(\cdot, P_j)\quad \hbox{ in }C^{2}_{loc}(\overline{ \Omega} \setminus \{P_1, \dots, P_m\})
\end{equation}
and
\begin{equation}\label{5}
\nabla \mathcal{KR}_\Omega (P_1,\dots ,P_m)=0.
\end{equation}

\end{theorem}

Next we describe the $local$ behavior of $u_n$ near the blow up points.

Let $\{ P_1, \cdots,P_m\}\in\S$ (see \eqref{2a}) and take $0<R\ll 1$ satisfying $\bar B_{2R}(P_i)\subset \Omega$ for $i=1,\dots,m$ and $B_R(P_i)\cap B_R(P_j)=\emptyset$ if $i\neq j$. For each $P_j\in \S$, $j=1,\dots,m$, there exists a sequence $\{x_{j,n}\}\in B_R(P_j)$ such that
\begin{itemize}
\item[$ i)$] $\un(x_{j,n})=\sup_{B_R(x_{j,n})}\un(x)\rightarrow +\infty$,
\item[$ii)$] $x_{j,n}\to P_j$ as $n\to +\infty$.
\end{itemize}
Then we rescale $\un$ around $x_{j,n}$ as
\begin{equation}\label{2.4}
\tilde u_{j,n}(x):= \un\left( \d_{j,n} x +x_{j,n}\right)- \un(x_{j,n})\quad \hbox{ in }B_{\frac{R}{\d_{j,n}}}(0)
\end{equation}
where the scaling parameter $\d_{j,n}$ is determined by
\begin{equation}\label{2.5}
\ln e^{\un(x_{j,n})}\d_{j,n}^2=1.
\end{equation}

By \cite[Corollary 4.3]{grohsu} there exists a constant $d_j>0$ such that, up to a sub-sequence,
\begin{equation}\label{2.6b}
\d_{j,n}=d_j \ln^{\frac 12}+o\left( \ln^{\frac 12}\right)
\end{equation}
as $n\to \infty$. Then \eqref{2.5} and \eqref{2.6b} in turn give
\begin{equation}\label{2.6c}
\un(x_{j,n})=-2 \log \ln +O(1)
\end{equation}
as $n\to \infty$ for any $j=1,\dots,m$.\\
The rescaled function $\tilde u_{j,n}$ in (\ref{2.4}) satisfies
\begin{equation}\nonumber
\left\{
\begin{array}{ll}
-\Delta \tu_{j,n}=e^{\tu_{j,n}} & \hbox{ in }B_{\frac{R}{\d_{j,n}}}(0)\\
\tu_{j,n}\leq \tu_{j,n}(0)=0& \hbox{ in }B_{\frac{R}{\d_{j,n}}}(0)
\end{array}
\right.
\end{equation}
and then a classification result (see \cite{cli}) implies
\begin{equation}\label{2.6}
\tu_{j,n}(x)\rightarrow U(x)=\log \frac 1{\left( 1+\frac{|x|^2}8\right)^2} \quad \hbox{ in }C^{\infty}_{loc}(\R^2)
\end{equation}
where $U(x)$ is the unique solution to the Liouville problem 
\[\begin{cases}
-\Delta U=e^U & \text{ in }\R^2\\
\int_ {\R^2} e^U <\infty.
\end{cases}\]
Moreover (see \cite{liy}), it holds that
\begin{equation}\label{2.6a}
\big| \tu_{j,n}(x)- U(x)\big|\leq C, \quad \forall x \in B_{\frac{R}{\d_{j,n}}}(0)
\end{equation}
with a constant $C>0$.
Finally, using \eqref{2.5} and \eqref{2.6b}, the following estimate holds 
 (see (6.7) in \cite{cli2} with $h(x)=1$ in $\Omega$ and  ${\bm{\lambda_k}}=\underbrace{u_n(x_{j,n})+\log \lambda_n+\log 8\pi m}_{=-\log\l_n+O(1)}$ for some $j$). 
\begin{equation}\label{stima2}
\lambda_n\int _{B_R(x_{j,n})}e^{ u_{n}(y)}dy=8\pi+O\big(\l_n\big).
\end{equation}
We end this section recalling the following nondegeneracy result (see Theorem 1.2 in \cite{grohsu} and \cite{gg1} for $m=1$). It will be useful in the proof of Theorem \ref{ex}.
\begin{theorem}\label{T5}
Let $u_n$
be  a sequence of solutions to \eqref{1} which blows-up at $\{P_1,\dots,P_m\}$ as $ n\to \infty$.
Suppose that $(P_1,\dots,P_m)$ is a nondegenerate critical point of the Kirchhoff-Routh function $\mathcal{KR}$.  Then $u_n$ is nondegenerate, i.e. the linearized problem
\begin{equation}
\label{l1}
            \left\{\begin{array}{lc}
                        -\Delta v=\l_ne^{u_n}v  &
            \mbox{  in }\Om\\
              v=0 & \mbox{ on }\partial \Om,
                      \end{array}
                \right.
\end{equation}
admits only the trivial solution $v\equiv0$ for $n$ large enough.
\end{theorem}

\subsection{Remarks on differential topology}\label{ss2}

We start this section recalling the celebrated Poincaré-Hopf Theorem which we state in the particular case where $\Omega$ is a bounded domain of $\R^N$. Note that if $v$ is a smooth vector field we denote by $index_{z_j}(v)=deg (v, B_{\delta}(z_j),0)$ the Brower Degree of the vector field $v$ in the ball $B_{\delta}(z_j)$ for some $\delta$ small fixed.
 
\begin{theorem}[Poincaré-Hopf Theorem]\label{teo-hopf}
Let $\Omega \subset \R^N$, with $N\geq 2$ be a smooth bounded domain.  Let $v$ be a smooth vector field on $\Omega$ with isolated zeroes $z_1,\dots,z_l$ and such that $v(x)\cdot \nu<0$ for any $x\in \partial \Omega$ (here $\nu$ is the outward normal vector to $\partial \Omega$). Then we have the formula
\[\sum_{j=1}^l index_{x_j}(v)=(-1)^N\chi(\Omega)\]
where $index_{x_j}(v)=deg (v, B_\delta (x_j),0)$ is the Brower degree of the vector field $v$  in the ball $B_\delta(x_j)$ with a small fixed $\delta>0$  and $\chi(\Omega)$ is the Euler characteristic of $\Omega$.
\end{theorem}
When $v=\nabla u_n$ this Theorem gives a link between an analytic problem (such as to compute the critical points of solutions to \eqref{1}) and a topological invariant such as $\chi(\Omega)$ that describes the structure of $\Omega$. The Euler characteristic is intrinsic of the manifold or of the domain and it is independent on the triangulation that one can use to reconstruct a manifold in the study of imaging. 

Another result which plays a crucial role in our computation is the following one, (see \cite{am})
\begin{theorem}\label{AM}
Let $\Om$ be a bounded open set in the plane and let its
boundary $\de\Om$ be composed of $N$ simple closed curves $\G_1,..,\G_N$, $N\ge1$ of
class $C^{1,\a}$. Consider the solution $u\in C^1(\bar\Om)\cap C^2(\Om)$ of the Dirichlet problem
$$
\begin{cases}
\D u=0&in\ \Om\\
u=a_i&on\ \de\Om
\end{cases}
$$
where $a_1, .. , a_N$ are given constants. If $a_1, .. , a_N$ do not all coincide, then $u$ has isolated critical points
$z_1, .. , z_k$ in $\bar\Om$, with finite multiplicities $m_1, .. , m_k$, respectively, and the following identity holds:
$$\sum_{z_k\in\Om}m_k+\frac12\sum_{z_k\in\de\Om}m_k=N-2
$$
\end{theorem}
\begin{remark}\label{RAM}
The multiplicity $m_k$ of a critical point $P=(x_0,y_0)$ of an analytic function $u:\Om\subset\R^2\to\R$ can be easily defined using complex coordinates. Indeed if $z=x+iy$ (and $z_0=x_0+iy_0$) we have that $m_k$ is defined by the following representation formula,
\begin{equation}
\partial _z u(z)=(z-z_0)^{m_k}g(z)
\end{equation}
with $g$ analytic and $g(z_0)\ne0$. 
\vskip0.2cm
Moreover (see pag.569 in \cite{am}) the following formula holds:
\begin{equation}
index_z(\nabla u) =-m_k.
\end{equation}
A consequence is that the index of a critical point of a harmonic satisfies
\begin{equation}
index_z(\nabla u)\le-1.
\end{equation}
Finally if  $index_z(\nabla u)=-1$ (and then $m_k=1$ and
there exists a coordinate system in which $z_0$ and the function $u(z)$ can be written as $\partial _z u(z)=a(z-z_0)+o(|z-z_0|^2)$ in a neighborhood of $z_0$. This implies that $z_0$ is a nondegenerate saddle point.
\end{remark}
\vskip0.2cm
We end this section stating a result on the number of critical points of solutions in domains with small holes.  
\begin{theorem}[see \cite{grlu}]\label{th1-1}
Let $\Omega $ be a smooth bounded set of $\R^2$ with $x_0\in \Omega$.
Suppose that $v_\d$ is a positive solution to
\begin{equation}\label{v-delta}
            \left\{\begin{array}{lc}
                        -\Delta v_\d=\l e^{v_\d}&
            \mbox{  in }\Omega\setminus B_\d(x_0)\\
              v_\d=0 & \mbox{ on }\partial \Omega \cup \partial B_\d(x_0)
                      \end{array}
                \right.
\end{equation}
which verifies, for some constant $C$ independent od $\d$,
\begin{equation}\label{limcr}
         v_\d \le C, \text{ in }\Omega\setminus B_\d(x_0).
\end{equation}
Denoting by $v_0$ the  weak limit of $v_\d$ as $\d\to 0$ we get that, if $x_0$ is not a critical point of $v_0$
and all critical points of $v_0$ are nondegenerate, 
\begin{equation}\label{mainresult}
\sharp\{\hbox{critical points of $v_\d$ in $\Omega\setminus B_\d(x_0)$}\}=\sharp\{\hbox{critical points of $v_0$ in $\Omega$}\}+1.
\end{equation}
Finally the additional critical point $x_\d$ of $v_\d$ is a saddle point of index $-1$.
\end{theorem}
In order to verify the assumption \eqref{limcr} we will use the following result,
{
\begin{lemma}\label{lemma-piccolezza}
Suppose that $u_\d$ is a family of solutions to \eqref{v-delta} in $\Omega\setminus B_\d(x_0)$. Assume $B_\d(x_0)\subset B_R(x_0) \subset \Omega$ and 
$$u_\d\leq C_{R}\hbox{ on }\partial B_R(x_0)$$
where $C_R$ is a constant independent of $\d$. Extending
$u_\d$ to zero in $B_\d(x_0)$ suppose that, for every $\d$,
\begin{equation}\label{epsilon-zero}
\l\int_{B_R(x_0)} e^{u_\d} dx\leq \e_0<4\pi.
\end{equation}
Then $\|u_\d\|_{L^{\infty}(B_R(x_0))}\leq C$ uniformly with respect to $\d$.
\end{lemma}
\begin{proof}
We adapt the proof of Corollary 3 and Theorem 1 in \cite{bremer} to our case where the domains are moving. For $x,y\in B(x_0,R)$ let us consider the function
\[\bar u_\d(x):=\frac 1{2\pi}\int_{B_R(x_0)}\log \frac {2R}{|x-y|}\left(\lambda e^{u_\d(y)}\right) dy\]
which solves
\begin{equation}\label{bar-u}
            \left\{\begin{array}{lc}
                        -\Delta \bar u_\d=\l e^{u_\d}&
            \mbox{  in }\R^2\\
              \bar u_\d\geq 0 & \mbox{ in } B_R(x_0).
                      \end{array}
                \right.
\end{equation}
The maximum principle, applied to the function $\bar u_\delta+C_{R}-u_\d$ in $B_R(x_0)\setminus B_\d(x_0)$, implies that 
\[u_\d\leq \bar u_\delta+C_{R} \ \ \text{ in }B_R(x_0).\]
The Jensen inequality (see the proof of Theorem 1 in \cite{bremer}) then gives, for every $0<\beta<4\pi$
\[e^{\int _{B_R(x_0)}\frac {\l e^{u_\d(y)}}{\|\l e^{u_{\d}}\|_{L^1(B_R(x_0))}}\frac {\beta}{2\pi} \log \left(\frac {2R}{|x-y|}\right) dy}\leq 
\int _{B_R(x_0)}\frac {\l e^{u_\d(y)}}{\|\l e^{u_{\d}}\|_{L^1(B_R(x_0))}}e^{\log \left(\frac {2R}{|x-y|}\right) ^{\frac \beta{2\pi}}}dy
\]
so that
\[e^{\frac \beta{\|\l e^{u_{\d}}\|_{L^1(B_R(x_0))}}\bar u_\d(x)}\leq \int _{B_R(x_0)}\frac {\l e^{u_\d(y)}}{\|\l e^{u_{\d}}\|_{L^1(B_R(x_0))}}\left(\frac {2R}{|x-y|}\right) ^{\frac \beta{2\pi}} dy
\]
which also implies,
\[e^{\frac \beta{\|\l e^{u_{\d}}\|_{L^1(B_R(x_0))}} u_\d(x)}\leq \int _{B_R(x_0)}\frac {\l e^{u_\d(y)}}{\|\l e^{u_{\d}}\|_{L^1(B_R(x_0))}}\left(\frac {2R}{|x-y|}\right) ^{\frac \beta{2\pi}} dy+e^{\frac \beta{\|\l e^{u_{\d}}\|_{L^1(B_R(x_0))}} C_{R}}.
\]
Integrating then
\[\begin{split}
\int _{B_R(x_0)}e^{\frac \beta{\|\l e^{u_{\d}}\|_{L^1(B_R(x_0))}} u_\d(x)}dx &\leq \int_{B_R(x_0)}dx \int _{B_R(x_0)}\frac {\l e^{u_\d(y)}}{\|\l e^{u_{\d}}\|_{L^1(B_R(x_0))}}\left(\frac {2R}{|x-y|}\right) ^{\frac \beta{2\pi}} dy\\
&+e^{\frac \beta{\|\l e^{u_{\d}}\|_{L^1(B_R(x_0))}} C_{R,\d}}\left| B_R\right|\\
&=\int_{B_R(x_0)} \frac {\l e^{u_\d(y)}}{\|\l e^{u_{\d}}\|_{L^1(B_R(x_0))}}dy \int _{B_R(x_0)}\left(\frac {2R}{|x-y|}\right) ^{\frac \beta{2\pi}} dx\\
&+(\pi R^2)e^{\frac \beta{\|\l e^{u_{\d}}\|_{L^1(B_R(x_0))}} C_{R}}.
\end{split}
\]
For $y\in B_R(x_0)$ we have that 
\[\int _{B_R(x_0)}\left(\frac {2R}{|x-y|}\right) ^{\frac \beta{2\pi}} dx\leq \int _{B_{2R}(y)}\left(\frac {2R}{|x-y|}\right) ^{\frac \beta{2\pi}} dx=\frac{8\pi R^2}{2-\frac \beta{2\pi}}.
\]
Then 
\[
\int _{B_R(x_0)}e^{\frac \beta{\|\l e^{u_{\d}}\|_{L^1(B_R(x_0))}} u_\d(x)}dx\leq \frac{8\pi R^2}{2-\frac \beta{2\pi}}+(\pi R^2)e^{\frac \beta{\|\l e^{u_{\d}}\|_{L^1(B_R(x_0))}} C_{R}}\leq C
\]
We choose $\beta$ such that $4\pi >\beta>\e_0$ so that $\frac \beta{\|\l e^{u_{\d}}\|_{L^1(B_R(x_0))}}>\frac \beta {\e_0}>1$. 
This means that $\lambda e^{u_\d}\in L^p$ for some $p>1$. 
Then we can apply the standard regularity theory to the function $\bar u_\d$ 
in \eqref{bar-u}, and we have that $\bar u_\d\in L^\infty(B_R(x_0))$, uniformly with respect to $\delta$. Then finally also $u_\d$ is uniformly bounded in $B_R(x_0)$.
\end{proof}
}

\sezione{An asymptotic estimate in an annular set that shrinks slowly }\label{S2}

From the known results in Section \ref{se:2}  we have good asymptotic for the solutions $u_n$ both in the blow-up balls $B_{\d_{i_n}\bar R}(x_{i,n})$ for every $\bar R>0$ (formulas \eqref{2.4}-\eqref{2.6a}) and in the set $\Om\setminus\cup_{j=1}^kB_R(P_j)$ for  any
$R>0$ sufficiently small (Theorem \ref{T0}).  It lasts to consider the case when a sequence of points $x_n\to P_i$ as $n\to \infty$ for some index $i\in \{1,\dots,m\}$ but $x_n$ do not belong to a ball of blow-up.  

This is the aim of this section, where we prove an asymptotic estimate in annuli centered at $x_{i,n}$ with arbitrary (infinitesimal) radii.

Before to give the precise statement, let us give an idea of the proof.

If we rescale the solutions $u_n$ near $x_{i,n}$ with respect to a parameter $r_n$ which is slower that $\delta_{i,n}$ in \eqref{2.5} we obtain, in the limit,  a problem that remains singular in the origin.  However, some properties of this "wrong" scaling will be useful for studying the properties of the solutions $u_n$ in points $x_n$ that converge to some $P_i$ but do no belong to $B_{\d_{i_n}\bar R}(x_{i,n})$ for every $\bar R$.
Here we include one of these results on this scaling of the solution $u_n$ that will be useful later. This idea has been already used in a previous paper \cite{gg1} in the case of a one-point blowing-up sequence of solutions $u_n$.

\begin{proposition}\label{prop-2.1}
Let $u_n$
be  a sequence of solutions to \eqref{1} which blows-up at $\{P_1,\dots,P_m\}$ as $ n\to \infty$ and let $r_n$  a sequence such that $r_n\to 0$ and 
\begin{equation}\label{b10}
\frac{r_n}{\delta_{i,n}}\to+\infty\quad\hbox{as }n\to \infty
\end{equation}
for every $i=1, \dots,m$, where $\delta_{i,n}$ are defined in \eqref{2.5}. 

Then,  for any $i=1,..,m$, and for any $0<\xi_1<\xi_2$, 
\begin{equation}\label{convergenza-vi}
u_n(x)=-4\log|x-x_{i,n}|+8\pi \mathcal{R}(P_i)+8\pi \sum_{j=1, j\neq i}^{m}G(P_i,P_j)+o(1)\quad\hbox{in }C^1_{loc}(A^i_n)
\end{equation}
as $n\to \infty$, where $A^i_n=\{\xi_1r_n\le|x-x_{i,n}|\le \xi_2r_n\}$ are shrinking annuli.
\end{proposition}
\begin{proof}
We follow an idea already used in \cite{gg1} in the case $m=1$. For some index $i\in\{1,..,m\}$ let us consider the function $v_{i,n}:\Om_n=\frac{\Om -x_{i,n}}{r_n}$ (observe that  $\Omega_n\to \R^2$ as $n\to \infty$),
\begin{equation}\label{vi}
v_{i,n}(x):= u_n(r_nx+x_{i,n})+4\log r_n.
\end{equation}
In this setting the estimate \eqref{convergenza-vi} is equivalent to show that
\begin{equation}\label{3.5}
v_{i,n}( x)=4\log \frac1{|x|}+8\pi \mathcal{R}(P_i)+8\pi \sum_{j=1, j\neq i}^{m}G(P_i,P_j)+o(1)\quad\hbox{in }C^1_{loc}(\R^2\setminus\{0\}).
\end{equation}\begin{equation}\label{3.5}
v_{i,n}( x)=4\log \frac1{|x|}+8\pi \mathcal{R}(P_i)+8\pi \sum_{j=1, j\neq i}^{m}G(P_i,P_j)+o(1)\quad\hbox{in }C^1_{loc}(\R^2\setminus\{0\}).
\end{equation}
To prove \eqref{3.5} 
we use the Green representation formula and we get
\begin{equation}\label{Avi}
\begin{split}
&v_{i,n}( x)=\lambda_n \int_{\Omega} G(r_n x+x_{i,n},y) e^{u_n(y)}dy+4\log r_n=\\
=&\underbrace{\lambda_n\int_{\Omega\setminus\left(\cup_j B_R(x_{j,n})\right)}G(r_n  x+x_{i,n},y) e^{u_n(y)}dy}_{=I_{1,n}}+\underbrace{\lambda_n\sum_{j=1}^n \int_{B_ R(x_{j,n})} G(r_n x+x_{i,n},y) e^{u_n(y)}dy}_{=I_{2,n}}+4\log r_n
\end{split}
\end{equation}
where $R$ is chosen as mentioned just before \eqref{2.4}.
From now we assume that
\begin{equation}\label{g4}
\xi_1\le| x|\le \xi_2.
\end{equation}
We start by proving $C^0$ estimates. We split the proof in some steps.
\vskip0.2cm
{\bf Step 1: estimate of $I_{1,n}$}
\vskip0.2cm
In this step we prove that, for $n\to \infty$
\begin{equation}\label{g2}
I_{1,n}=o(1).
\end{equation}
Indeed in $\Omega\setminus\left(\cup_j B_R(x_{j,n})\right)$, by \eqref{3} $u_n$ is uniformly bounded and since $| x|\le \xi_2$ we have
\[
I_{1,n}=\lambda_n\int_{\Omega\setminus\left(\cup_j B_R(x_{j,n})\right)}G(r_n x+x_{i,n},y) e^{u_n(y)}dy=O(\lambda_n)=o(1)\]
as $n\to \infty$, which gives \eqref{g2}.
\vskip0.2cm
{\bf Step 2: estimate of $I_{2,n}$}
\vskip0.2cm
In this step we prove that, for $n\to \infty$
\begin{equation}\label{g3}
I_{2,n}=-4\log r_n+ 4\log \frac 1{| x|}+8\pi \mathcal{R}(P_i)+8\pi \sum_{j=1, j\neq i}^{m}G(P_i,P_j)+o(1).
\end{equation}
This is the more delicate estimate and we need to consider the different cases $j\ne i$ and $j=i$.
\vskip0.2cm
{\bf Step 3: case $j\ne i$}
\vskip0.2cm
When $j\neq i$ it is easy to see that $|G(r_n x+x_{i,n},y)|\leq C$ for $y\in B_R(x_{j,n})$ since $r_n x+x_{i,n}\in B_R(x_{i,n})$ for $n$ large (by $| x|\le \xi _2$) and so
\[\begin{split}
&\lambda_n\int_{B_R(x_{j,n})} G(r_n x+x_{i,n},y) e^{u_n(y)}dy=\int_{B_{\frac R{\delta_{j,n}}}(0)}G(r_n x+x_{i,n},\delta_{j,n} y+x_{j,n})e^{ \tilde u_{j,n}( y)}d y=\\
& \ \ \ \ 8\pi G(P_i,P_j)+o(1)
\end{split}\]
when $n\to \infty$, by \eqref{2.6a}. 
\vskip0.2cm
{\bf Step 4: case $j=i$}
\vskip0.2cm
Here we have that
\[\begin{split}
&\lambda_n \int_{B_R(x_{i,n})} G(r_n  x+x_{i,n},y) e^{u_n(y)}dy\\
& \ \ \ =\lambda_n\int _{B_R(x_{i,n})} \left[\frac 1{2\pi}\log \frac 1{|r_n x+x_{i,n}-y|}+H(r_n x+x_{i,n},y)\right] e^{u_n(y)}dy\\
& \ \ \ \ \text{ and, letting }y=\delta_{i,n} y+x_{i,n}\\
& \ \ =\frac 1{2\pi}\int_{B_{\frac R{\delta_{i,n}}}(0)}\log \frac 1{|r_n x-\delta_{i,n} y|}e^{ \tilde u_{i,n}( y)}d y+ \int _{B_{\frac R{\delta_{i,n}}}(0)} H(r_n x+x_{i,n},\delta_{i,n} y+x_{i,n})e^{ \tilde u_{i,n}( y)}d y=\\
& \ \ =\frac 1{2\pi}\int_{B_{\frac R{\delta_{i,n}}}(0)}\log \frac 1{| x-\frac{\delta_{i,n}}{r_n} y|}e^{ \tilde u_{i,n}( y)}d y+ \frac 1{2\pi}\log \frac 1{r_n}\int_{B_{\frac R{\delta_{i,n}}}(0)}e^{ \tilde u_{i,n}( y)}d y\\
& \ \ \ \ \ \ \ \ +
\int _{B_{\frac R{\delta_{i,n}}}(0)} H(r_n x+x_{i,n},\delta_{i,n} y+x_{i,n})e^{ \tilde u_{i,n}( y)}d y=L_{1,n}+L_{2,n}+L_{3,n}
\end{split}\]
The easiest integral to estimate is $L_{3,n}$. Indeed, since $|H(r_n x+x_{i,n},\delta_{i,n} y+x_{i,n})|\leq C$ for $ y\in  B_{\frac R {\delta_{i,n}}(0)}$ by \eqref{2.6a} then, as $n\to \infty$
\[L_{3,n}=\int _{B_{\frac R{\delta_{i,n}}}(0)} H(r_n x+x_{i,n},\delta_{i,n} y+x_{i,n})e^{ \tilde u_{i,n}( y)}d y= 8\pi H(P_i,P_i)+o(1)=8\pi \mathcal{R}(P_i)+o(1).\]
Next we prove that,  as $n\to \infty$
\begin{equation}\label{g1}
L_{2,n}=-4\log r_n+O\big(\l_n^2\log r_n\big)=-4\log r_n+o(1).
\end{equation}
Indeed by \eqref{stima2} we have 
\[\begin{split}
L_{2,n}&=\frac 1{2\pi}
\log \frac 1{r_n}\int_{B_{\frac R{\delta_{i,n}}}(0)}e^{ \tilde u_{i,n}( y)} d y=\frac{8\pi+O\big(\l_n\big)}{2\pi}
\log \frac 1{r_n}\\
&=-4\log r_n+O(\delta_{i,n}^2\log r_n)
\end{split}
\]
as $n\to \infty$. Finally  by \eqref{b10} we get $\delta_{i,n}^2\log r_n=o(1)$ and \eqref{g1} follows.\\
Finally let us estimate the integral
\[
L_{1,n}:=\frac 1{2\pi}\int_{B_{\frac R{\delta_{i,n}}}(0)}\log \frac 1{| x-\frac{\delta_{i,n}}{r_n} y|}e^{\tilde  u_{i,n}( y)}d y.
\]
Take $\rho=\frac {\xi _1}2>0$ and let us split $B_{\frac R{\delta_{i,n}}(0)}$ as $$B_{\frac R{\delta_{i,n}}(0)}=B_{\rho\frac {r_n}{\delta_{i,n}}}\left(\frac{r_n}{\delta_{i,n}} x\right)\cup\left(B_{\frac R{\delta_{i,n}}(0)}\setminus B_{\rho\frac {r_n}{\delta_{i,n}}}\left(\frac{r_n}{\delta_{i,n}} x\right)\right).$$
Observe that by \eqref{b10} and \eqref{g4} and using that $|x|\geq \xi_1$, if  $ y\in B_{\rho\frac {r_n}{\delta_{i,n}}}\left(\frac{r_n}{\delta_{i,n}} x\right)$ then $| y|=|\frac{r_n}{\delta_{i,n}} x+ y-\frac{r_n}{\delta_{i,n}} x|\geq \frac{r_n}{\delta_{i,n}}\xi _1-\rho \frac{r_n}{\delta_{i,n}}=\frac{r_n}{\delta_{i,n}}\frac{\xi_1}2\to \infty$ if $n\to \infty$. This means that
$$B_{\frac R{\delta_{i,n}}(0)}\setminus B_{\rho\frac {r_n}{\delta_{i,n}}}\left(\frac{r_n}{\delta_{i,n}} x\right)\to \R^2
$$
as $n\to \infty$.
Then we write
\[\begin{split}
L_{1,n}=&\frac 1{2\pi} \int_{B_{\frac R{\delta_{i,n}}(0)}\setminus B_{\rho\frac {r_n}{\delta_{i,n}}}\left(\frac{r_n}{\delta_{i,n}} x\right)}
\log \frac 1{| x-\frac{\delta_{i,n}}{r_n} y|}
e^{ \tilde u_{i,n}( y)}d y\\
&+\frac 1{2\pi} \int_{ B_{\rho\frac {r_n}{\delta_{i,n}}}\left(\frac{r_n}{\delta_{i,n}} x\right)}
\log \frac 1{| x-\frac{\delta_{i,n}}{r_n} y|}
e^{ \tilde u_{i,n}( y)}d y:=J_{1,n}+J_{2,n}.
\end{split}
\]
In $J_{1,n}$, since $ y\notin B_{\rho\frac {r_n}{\delta_{i,n}}}\left(\frac{r_n}{\delta_{i,n}} x\right)$ we have that $| x-\frac{\delta_{i,n}}{r_n} y|=\frac{\delta_{i,n}}{r_n}|\frac{r_n}{\delta_{i,n}} x- y|>\rho$. Then
\[\left| \log \frac 1{| x-\frac{\delta_{i,n}}{r_n} y|}\right|\leq \max \{|\log \rho|, | x-\frac{\delta_{i,n}}{r_n} y|\}\leq C+| x|+| y|\leq C+| y|.
\]
Then, by \eqref{2.6a} we have
\[\log \frac 1{| x-\frac{\delta_{i,n}}{r_n} y|}
e^{ \tilde u_{i,n}( y)}\leq \frac{ C+| y|}{(1+| y|^2)^2}
\]
and passing to the limit we get, as $n\to \infty$
\[J_{1,n}=\frac 1{2\pi}\log \frac {1}{| x|}\int_{\R^2}e^{U( y)}d y+o(1)=4 \log \frac {1}{| x|}+o(1).\]
Next we estimate $J_{2,n}$. We use \eqref{2.6a},  $| y|\geq \frac{r_n}{\delta_{i,n}}\frac{\xi _1}2$ and $\frac {r_n}{\delta_{i,n}}\to+\infty$ to have that 
\begin{equation}\label{numero-bis}
e^{ \tilde u_{i,n}(y)}\leq \frac C{\left(8+\frac {r_n^2}{\delta_{i,n}^2}\frac{\xi _1^2}{4}\right)^2}\leq C \frac {\delta_{i,n}^4}{r_n^4}.
\end{equation}
Then
\[\begin{split}
&|J_{2,n}|
\leq\frac 1{2\pi} \int_{ B_{\rho\frac {r_n}{\delta_{i,n}}}\left(\frac{r_n}{\delta_{i,n}} x\right)}\left|\log 
\frac {1}{| x-\frac{\delta_{i,n}}{r_n} y|}\right| e^{ \tilde u_{i,n}( y)}d y\\
& \ \ \  
\leq C \frac {\delta_{i,n}^4}{r_n^4}\int_{ B_{\rho\frac {r_n}{\delta_{i,n}}}\left(\frac{r_n}{\delta_{i,n}} x\right)}
\left|\log 
\frac {1}{| x-\frac{\delta_{i,n}}{r_n} y|}\right| 
d y
\\
& \ \ \ =C \frac {\delta_{i,n}^4}{r_n^4}\int_0^{ \rho\frac {r_n}{\delta_{i,n}}} r
\left|\log\left(r\frac{\delta_{i,n}}{r_n}\right)\right| 
dr=O\left(\frac {\delta_{i,n}^2}{r_n^2}\log\frac {\delta_{i,n}}{r_n}\right)=o(1)
\end{split}
\]
as $n\to \infty$.
Summarizing, we have as $n\to \infty$
\[\begin{split}
I_{2,n}=&\lambda_n\sum_{j=1}^n \int_{B_ R(x_{j,n})} G(r_n x+x_{i,n},y) e^{u_n(y)}dy\\
&=\sum_{j=1, j\neq i}^{m} \underbrace{\lambda_n \int_{B_ R(x_{j,n})} G(r_n x+x_{i,n},y) e^{u_n(y)}dy}_{=8\pi G(P_i,P_j)+o(1)}+\underbrace{L_{1,n}}_{J_{1,n}+J_{2,n}}+\underbrace{L_{2,n}}_{-4\log r_n+o(1)}+\underbrace{L_{3,n}}_{8\pi \mathcal{R}(P_i)+o(1)}\\
&\text{ and using the expansions for $J_{1,n}$ and $J_{2,n}$}\\
&=-4\log r_n+4\log \frac 1{| x|}+8\pi \mathcal{R}(P_i)+8\pi \sum_{j=1, j\neq i}^{m}G(P_i,P_j)+o(1)
\end{split}\]
which gives the claim in \eqref{g3}.
\vskip0.2cm
{\bf Step 5, claim of $C^0$-estimate}
\vskip0.2cm
Putting together all the estimates of the previous steps we get
\[\begin{split}
v_{i,n}( x)
&=\lambda_n\int_{\Omega\setminus\left(\cup_j B_{R}(x_{j,n})\right)}G(r_n x+x_{i,n},y) e^{u_n(y)}dy+\lambda_n\sum_{j=1}^m \int_{B_{R}(x_{j,n})} G(r_n x+x_{i,n},y) e^{u_n(y)}dy
\\
&+4\log r_n=I_{1,n}+I_{2,n}+4\log r_n\\
&=4\log \frac 1{| x|}+8\pi \mathcal{R}(P_i)+8\pi \sum_{j=1, j\neq i}^{m}G(P_i,P_j) +o(1)
\end{split}\]
as $n\to \infty$ uniformly for $x\in \{x\in \R^2: \xi_1\le x\le\xi_2\}$ which gives the claim of the $C^0$-estimate in \eqref{3.5} and also in \eqref{convergenza-vi} .
\vskip0.2cm
{\bf Step 6, claim of $C^1$-estimate}
\vskip0.2cm
We can do the same computations for the derivatives of $v_{i,n}$.  Again from the Green representation formula and equation \eqref{1}  we have
\begin{equation}\label{numero}
\begin{split}
\nabla v_{i,n}(x)&=r_n\lambda_n \int_{\Omega} \nabla_x G(r_n x+x_{i,n},y) e^{u_n(y)}dy=r_n\lambda_n\int_{\Omega\setminus\left(\cup_j B_{R}(x_{j,n})\right)}\nabla_xG(r_n x+x_{i,n},y) e^{u_n(y)}dy+\\
&r_n \lambda_n\sum_{j=1}^n \int_{B_{R}(x_{j,n})} \nabla_x G(r_n x+x_{i,n},y) e^{u_n(y)}dy
\end{split}\end{equation}
where, as before 
\[\lambda_n\int_{\Omega\setminus\left(\cup_j B_{R}(x_{j,n})\right)}\nabla_xG(r_n x+x_{i,n},y) e^{u_n(y)}dy=O(\lambda_n)=o(1)\]
as $n\to \infty$ and, for $i\neq j$ 
\[\begin{split}
r_n\lambda_n\int_{B_{R}(x_{j,n})} \nabla_x G(r_n x+x_{i,n},y) e^{u_n(y)}dy&=r_n \int_{B_{\frac R{\delta_{j,n}}(0)}}\nabla_x G(r_n x+x_{i,n},\delta_{j,n}y+x_{j,n}) e^{\tilde u_{j,n}( y)}d y\\
&=r_n  8\pi \nabla _x G(P_i,P_j)+o(r_n)=o(1).\end{split}\]
Next we consider the last term in \eqref{numero}
\[\begin{split}
&r_n\lambda_n\int_{B_{R}(x_{i,n})} \nabla_x G(r_n x+x_{i,n},y) e^{u_n(y)}dy\\
&\ \ \ \ \ \ =r_n\lambda_n \int _{B_{ R}(x_{i,n})}\left(-\frac 1{2\pi}\frac {(r_n x+x_{i,n}-y)}{|r_n  x+x_{i,n}-y|^2}+\nabla_x H(r_n  x+x_{i,n},y)\right)e^{u_n(y)}dy\\
&\ \ \ \ \ \ =-\frac 1{2\pi} r_n\int_{B_{\frac R{\delta_{i,n}}(0)}}\frac {(r_n x-\delta_{i,n}y)}{|r_n x-\delta_{i,n} y|^2}e^{\tilde u_{i,n}( y)}dy+r_n\int_{B_{\frac R{\delta_{i,n}}(0)}}\nabla_x H(r_n  x+x_{i,n},\delta_{i,n}y+x_{i,n})e^{\tilde u_{i,n}( y)}dy
\\
&\ \ \ \ \ \ =-\frac 1{2\pi}\int_{B_{\frac R{\delta_{i,n}}(0)}}\frac {( x-\frac {\delta_{i,n}}{r_n} y)}{| x-\frac{\delta_{i,n}}{r_n}y|^2}e^{\tilde u_{i,n}( y)}d y+\underbrace{8\pi r_n \nabla_x H(P_i,P_i)+o(r_n)}_{=o(1)}
\end{split}\]
where we used again \eqref{2.6a} to pass to the limit.


As we did in the computation of $J_{1,n}$ and $J_{2,n}$ is Step $4$, we split the last integral as follows,
\[\begin{split}
I_n:=&-\frac 1{2\pi} \int_{B_{\frac R{\delta_{i,n}}(0)}}\frac {( x-\frac {\delta_{i,n}}{r_n} y)}{| x-\frac{\delta_{i,n}}{r_n}y|^2}e^{\tilde u_{i,n}( y)}d y\\
=&-\frac 1{2\pi} \int_{B_{\frac R{4\delta_{i,n}}(0)}\setminus B_{\rho\frac {r_n}{\delta_{i,n}}}\left(\frac{r_n}{\delta_{i,n}} x\right)}\frac {( x-\frac {\delta_{i,n}}{r_n} y)}{| x-\frac{\delta_{i,n}}{r_n} y|^2}e^{\tilde u_{i,n}( y)}dy\\
&-\frac 1{2\pi} \int_{ B_{\rho\frac {r_n}{\delta_{i,n}}}\left(\frac{r_n}{\delta_{i,n}} x\right)}\frac {( x-\frac {\delta_{i,n}}{r_n} y)}{| x-\frac{\delta_{i,n}}{r_n} y|^2}e^{\tilde u_{i,n}( y)}d y:=\bar I_{1,n}+\bar I_{2,n}
\end{split}
\]
In $\bar I_{1,n}$ 
$$\left|\frac {( x-\frac {\delta_{i,n}}{r_n} y)}{| x-\frac{\delta_{i,n}}{r_n} y|^2}\right|=
\frac1{| x-\frac{\delta_{i,n}}{r_n} y|}\le\frac1\rho\hbox{ since  } y\notin B_{\rho\frac {r_n}{\delta_{i,n}}}\left(\frac{r_n}{\delta_{i,n}} x\right)
$$
and, by \eqref{2.6a} we get, as $n\to \infty$
\[\bar I_{1,n}= -\frac 1{2\pi}\frac { x}{|x|^2}\int_{\R^2}e^{U(y)}d y+o(1)=-4 \frac { x}{| x|^2}+o(1).\]
In $\bar I_{2,n}$ instead we use \eqref{numero-bis} to get
\[\begin{split}
&|\bar I_{2,n}|=\frac 1{2\pi}\left|\int_{ B_{\rho\frac {r_n}{\delta_{i,n}}}\left(\frac{r_n}{\delta_{i,n}} x\right)}
\frac {(x-\frac {\delta_{i,n}}{r_n} y)}{| x-\frac{\delta_{i,n}}{r_n} y|^2}e^{\tilde u_{i,n}( y)}dy\right|\\
& \ \ \  \leq C \frac {\delta_{i,n}^4}{r_n^4}\int_{ B_{\rho\frac {r_n}{\delta_{i,n}}}\left(\frac{r_n}{\delta_{i,n}} x\right)}\frac {1}{| x-\frac{\delta_{i,n}}{r_n} y|} d y
\\
& \ \ \ =C \frac {\delta_{i,n}^4}{r_n^4} \int_0^{2\pi}d\theta\int_0^{ \rho\frac {r_n}{\delta_{i,n}}} \frac r{r\frac {\delta_{i,n}}{r_n}} dr=2 \pi C \frac {\delta_{i,n}^4}{r_n^4}\frac {r_n^2}{\delta_{i,n}^2}\to 0
\end{split}
\]
Putting together all the estimates we get
\begin{equation}\label{numero2}
\nabla v_{i,n}( x)=-4\frac{ x}{| x|^2}+o(1)
\end{equation}
uniformly for $x\in \{x\in \R^2: \xi_1\le x\le\xi_2\}$  
as $n\to \infty$ and this concludes the proof.

\end{proof}

\section{Critical points of $u_\l$ in $B_{\rho}(P_i)$}\label{S3}
In this section we will prove that any solution $u_\l$ which blows-up at $m$ points $\{P_1,\dots,P_m\}\in \Omega$, as $\lambda\to 0$, admits $exactly$ one non-degenerate critical point in each ball $B_\rho(P_i)$,  for $i=1,..,m$ if $\rho$ is small enough.
It is sufficient to prove the result for any sequence of values $\lambda_n$ such that $\lambda_n\to 0$ as $n\to \infty$.
\begin{proposition}\label{lem1}
Let $u_n$
be  a sequence of solutions to \eqref{1} which blows-up at $\{P_1,\dots,P_m\}$ as $n\to \infty$. Then there exists $\rho>0$ such that $u_n$ has a unique critical point in $\bar B_\rho(P_i)$ for $i=1,\dots,m$ when $n$ is large. This critical point is  the maximum $x_{i,n}$ in \eqref{2.4} and it is nondegenerate.
\end{proposition}
\begin{remark}\label{ultimo}
Actually we will prove something more, precisely that there exists $\bar \rho>0$ such that 
\begin{equation}\label{a1}
(x-x_{i,n})\cdot \nabla u_n(x)<0 \text{ in } B_{\bar \rho}(x_{i,n})\setminus\{x_{i,n}\}
\end{equation}
for $i=1,\dots,m$.  Then we take 
$\rho:=\frac{\bar \rho}2$, and we get that 
\[(x-x_{i,n})\cdot \nabla u_n(x)<0 \text{ in } \bar B_{\rho}(P_i)\setminus\{x_{i,n}\}.\]
This will show that $u_n$ has the unique critical point $x_{i,n}$ in $\bar B_{\rho}(P_i)$.  
\end{remark}
\vskip0.2cm
\begin{proof} 
As pointed out in Remark \ref{ultimo}, we prove \eqref{a1} to get that $u_n$ has the unique critical point $x_{i,n}$ in $\bar B_{\rho}(P_i)$.  \\
We argue by contradiction and we assume there exists an index $i\in \{1,\dots,m\}$ and points $\zeta_n\in B_{\bar \rho_n}(x_{i,n})\setminus\{x_{i,n}\}$ such that $\bar \rho_n\to 0$ and 
\[(\zeta_n-x_{i,n})\cdot \nabla u_n(\zeta_n)\geq 0  \ \ \text{ for every }n.\]
Up to a subsequence $\zeta_n\to \zeta$ and since $\bar \rho_n\to 0$  then $\zeta=P_i$.  Let  $\delta_{i,n}$ be the scaling parameter as defined in \eqref{2.5} related to the index $i$. Up to a subsequence, we have the following alternative,
\begin{itemize}
\item[a)]$\zeta_n\in B_{\bar R\d_n}(x_{i,n})$ for some $\bar R>0$ for any $n$.
\item[b)] $\zeta_n\notin B_{\delta_{i,n}\bar R}(x_{i,n})$ for any real number $\bar R>0$ for any $n$.
\end{itemize}
\vskip0.2cm
\underline{Proof of case a)}
\vskip0.2cm
Here we rescale the solution $u_n$ around $x_{i,n}$ as in \eqref{2.4} and we let $\tilde \zeta_n:=\frac{\zeta_n-x_{i,n}}{\delta_{i,n}}$.  Then $|\tilde \zeta_n|<\bar R$ and so, up to a subsequence $\tilde 
\zeta_n\to \tilde \zeta\in B_{\bar R}(0)$. From \eqref{2.4} we have that  $\tilde \zeta_n\cdot \nabla \tilde u_{i,n}(\tilde \zeta_n)=(\zeta_n-x_{i,n})\cdot \nabla u_n(\zeta_n)\geq 0$, and, passing to the limit and using \eqref{2.6} $\tilde \zeta_n\cdot \nabla \tilde u_{i,n}(\tilde \zeta_n)\to\tilde \zeta\cdot \nabla U(\tilde \zeta)=-\frac 12 \frac{|\tilde \zeta|^2}{\left(1+\frac {|\tilde \zeta|^2}8\right)}\leq 0$. If $\tilde \zeta\neq 0$ we reach a contradiction. 

So suppose that $\tilde \zeta=0$ and consider  the function $\varphi_n(t):=\tilde u_{i,n}(t\tilde \zeta_n)$. Since  $\varphi_n$ has a maximum at $0$ and $\varphi_n'(1)=\tilde \zeta_n\cdot \nabla \tilde u_{i,n}(\tilde \zeta_n)\ge0$ we deduce that there exists a $minimum$ point $z_n\in(0,\tilde \zeta_n)$. From $\varphi_n''(z_n)\ge0$, passing to the limit we get that 
$$\frac{\partial^2 U(0)}{\partial y_i\partial y_j}\xi_i\xi_j\ge0$$
with $\xi=\lim\limits_{n\to+\infty}\frac{z_n}{|z_n|}$. This gives a contradiction since $0$ is a $non-degenerate$ maximum of $U$.

This proves the uniqueness of the critical point of $u_n$ in $B_{\bar R\d_n}(x_{i,n})$ and the $C^2$ convergence of $\tilde u_{i,n}$ to $U$ gives also its non-degeneracy.

\vskip0.2cm
\underline{Proof of case b)}
\vskip0.2cm
In this case we have that $\zeta_n\to0$ with $\lim\limits_{n\to+\infty}\frac{|\zeta_n-x_{i,n}|}{\d_{i,n}}=+\infty$.
Here we use Proposition \ref{prop-2.1} with $r_n:=|\zeta_n-x_{i,n}|$.  
Let $\tilde\zeta_n:=\frac{\zeta_n-x_{i,n}}{r_n}$. So $|\tilde\zeta_n|=1$ for every $n$ and we have that, if $v_{i,n}$ is the function defined in \eqref{vi},
\[\tilde\zeta_n\cdot \nabla v_{i,n}(\tilde \zeta_n)=(\zeta_n-x_{i,n})\cdot \nabla u_n(\zeta_n)\geq 0.\]
Using \eqref{numero2} we have that, up to a subsequence, $\tilde\zeta_n\to \tilde\zeta\in \partial B_1(0)$ and $\tilde \zeta$ satisfies
\[\tilde\zeta\cdot \left(-4\frac{\tilde\zeta}{|\tilde\zeta|^2}\right)=-4 \geq 0\]
which gives a contradiction.

So from \eqref{a1} we get the claim.
\end{proof}
By the previous proposition the problem of the computation of the number of critical points of $u_\l$ is reduced to understand what happens in the region $\Omega\setminus \cup_{j=1}^m \bar B_\rho(P_j)$. We discuss the different situations in next section.

\section{Critical points of $u_\l$ in $\Omega\setminus \cup_{j=1}^m \bar B_\rho(P_j)$}\label{S4}

In this section we study the number of critical points of $u_\l$ in the set $$D:=\Omega\setminus \cup_{j=1}^m \bar B_\rho(P_j),$$
where $\rho$ is the value given in Proposition \ref{lem1}. As previously pointed out, the geometry and the topology fo $D$ have a great influence.

In the set $D$ it will important to consider the harmonic function 
\begin{equation}\label{kx}
K(x):=8\pi \sum_{j=1}^m G(x, P_j).
\end{equation}
We start by showing some of the properties of $K(x)$, they are basically a consequence of Theorem \ref{AM}.
\begin{proposition}\label{prop-critical-points-K}
The function $K(x)$ in $\Omega\setminus\{P_1,\dots,P_m\}$ has only a finite number of  critical points $\mathcal C:=\{z_1,\dots, z_l\}$ which are saddle points  of finite multiplicity $m_j \geq 1$ and $index_{z_j}(\nabla K) \leq -1$.  Moreover whenever $index_{z_j}(\nabla K) =-1$ then $z_j$ is a nondegenerate saddle point.
\end{proposition}
\begin{proof}
The function $K(x)$ is harmonic in $\Omega\setminus\{P_1,\dots,P_m\}$.  Since
\begin{itemize}
\item $\lim\limits_{x\to P_j}K(x),|\nabla K(x)|\to +\infty$ for $j=1,\dots,m$ 
\item $K(x), |\nabla K(x)|=O(1)$ in $\omega$ if $\bar \omega\subset \bar \Omega \setminus\{P_1,\dots,P_m\}$
\end{itemize}
we can select a value of $M$ large in such a way that the level set $K(x)=M$ is given by $m$ curves  $\Gamma_1, \dots, \Gamma_m$ such that $\Gamma_i\cap \Gamma_j=\emptyset$ for $i\neq j$.  

Moreover $M$ can be chosen such that $|\nabla K(x)|>0$ on the curves $\Gamma_j$ for $j=1,\dots,m$. 

Then these curves are simple, closed, smooth and are the boundary of bounded sets $A_j$ such that $P_j\in A_j$.\\
Since $K(x)=M>0$ on $\Gamma_1\cup\dots\cup\Gamma_m$ and $K(x)=0$ on $\partial \Omega$ can apply Theorem \ref{AM}  getting that $K(x)$ has a finite number of critical points in $\Omega\setminus\cup_{j=1}^m \bar A_j$. 
Moreover since $K(x)>0$ in $D$ the Hopf Lemma implies that there are not critical points on $\partial \Omega$. 

 The maximum principle for harmonic functions then implies that the critical points $\{z_1,\dots, z_l\}$ are saddle critical points with finite multiplicities $m_1,\dots,m_l$ by the analiticity of harmonic functions.  
 
Finally by Remark \ref{RAM} we have that if $index_{z_j}(\nabla K)=-1$ 
then $m_j=1$ and 
$z_j$ is a nondegenerate saddle point.
\end{proof}
As in the previous section it is enough to prove the results for any sequence of values $\lambda_n$ such that $\lambda_n\to 0$.
\begin{proposition}\label{prop2}
Let $u_n$ be  a sequence of solutions to \eqref{1} which blows-up at $\{P_1,\dots,P_m\}$ as $n\to \infty$.  
Then $u_n$,  for $n$ large enough, has only a finite number of isolated critical points that we denote by $\mathcal C_n:=\{z_{1,n},\dots,z_{l_n,n}\}$.  
Moreover 
\begin{equation}\label{ind}
index_{z_{j,n}}(\nabla u_n)\in \{1,0,-1\}
\end{equation}
and, whenever the index is $1$ then $z_{j,n}$ is a maximum, while whenever the index is $-1$ $z_{j,n}$ is a nondegenerate saddle point.
\end{proposition}
\begin{proof}
The proof uses some ideas by \cite{am}. By classical results  $u_n$ is real-analytic in $\Omega$ (see for example \cite{am},  Cor 3.4 and the references therein).  

By Lemma \ref{lem1} $u_n$ has a unique critical point in $\bar B_\rho(P_i)$ for $i=1,\dots,m$ when $n$ is large enough and it cannot have critical points on $\partial \Omega$ by Hopf lemma.
First we show that the critical points of $u_n$ in $D$ are isolated when $n$ is large.  By \eqref{3} the critical points of $u_n$ in $D$ should converge to a critical point of $K(x)$ as $n\to \infty$.  Let us argue by contradiction and suppose that there exists a critical point $z_0:=(x_0,y_0)\in D$ for the function $K(x)$ which is the limit of a sequence of critical points $z_n$ for $u_n$ and such that the points $z_n$ are not isolated for every $n$. Then there exists sequences of points $z_n^h\in D$ such that $\nabla u_n(z_n^h)=0$ for every $n$ and for every $h$ and such that $z_n^h\to z_n$ as $h\to \infty$ and $z_n^h\to z_0$ as $n\to \infty$. Using that $-\Delta u_n=\lambda_n e^{u_n}>0$  we may assume, up to a rotation, that $(u_n)_{yy}(z_n)\neq 0$. By the implicit function theorem, for every $n$  there exists a neighborhood of $z_n$ in which the set $\sigma_n:=\{x\in D: (u_n)_y=0\}$ is a simple analytic arc that contains infinitely many of the points $z_n^h$. Since $(u_n)_x(z_n^h)=0$ for infinitely many points $z_n^h$ then it should be (by analicity) $(u_n)_x(x)=0$ on $\sigma_n$ meaning that the set of critical points of $u_n$ that passes in $z_n$ is a curve, that we call $\gamma_n$.  
These curves $\gamma_n$ are closed and contained in $D$ for every $n$ large.  (This follows since $u_n$ cannot have critical points on $\partial D$ and a curve of critical points cannot end inside $D$ by the maximum principle).\\

We denote by $G_n$ the subset of the plane that is bounded and such that $\partial G_n=\gamma_n$.  Observe that $\partial G_n=\gamma_n$ is smooth at least in a neighborhood of $z_n$ by construction. It is not possible that $G_n\subset D$. Indeed $u_n=c_n=costant$ on $\gamma_n$ while $-\Delta u_n>0$ in $G_n$  imply that $u_n(x)>c_n$ in $G_n$.  Then the Hopf lemma (that we can apply at  least in $z_n$) implies that the normal derivative of $u_n$ on $\gamma_n$ is negative, while $\gamma_n$ being a curve of critical points for $u_n$ forces $\nabla u_n(x)\cdot \nu=0$ for every $x\in \gamma_n$. 

Then, for every $n$ the set $G_n$ should contain a hole of $D$ (namely either a hole of $\Omega$ or a hole $B_\rho(P_j)$). Up to a subsequence the sets $G_n$ contain the same hole $\mathcal O\subset \R^2$ (that does not depend on $n$) for every $n$. 
The fact that $\mathcal O\subset G_n$ implies that, in the limit as $n\to \infty$, $\gamma_n$ converges to a closed curve $ \gamma$ which is the boundary of the nonempty set $G$,  and which,  by \eqref{3} is a curve of critical points for the function $K(x)$.  But this is not possible by Proposition \ref{prop-critical-points-K}. 
Then the critical points of $u_n$ are isolated for $n$ large. Finally, since $D$ is compact, then the set of critical points of $u_n$ is finite. The classification of the type of critical points for $u_n$ and formula \eqref{ind} then follows by Theorem 3.3 in \cite{am}.
\end{proof}
Now we are able to use the Poincarè Hopf Theorem to get: 
\begin{proposition}\label{prop-preuguali}
Let $u_n$
be  a sequence of solutions to \eqref{1} which blows-up at $\{P_1,\dots,P_m\}$ as $n\to \infty$.  Then, when $n$ is large
\begin{equation}\label{prima}
m+\sum_{z_j\in \mathcal C_n} index_{z_j} (\nabla u_n)=\chi(\Omega)
\end{equation}
where $\mathcal C_n$ is the set of critical points of $u_n$ in $D$ and $\chi(\Omega)$ is the Euler characteristic of $\Omega$.
\end{proposition}
\begin{proof}
Since the critical points of $u_n$ are isolated and finite, when $n$ is large,  by Proposition \ref{prop2} we can use the Poincarè Hopf formula (see Theorem \ref{teo-hopf}), with $v=\nabla u_n$ in $\Omega$. Observe that by Hopf Lemma we have that $\nabla u_n \cdot \nu<0$. Then
\[\sum_{z_j} index_{z_j}(\nabla u_n)=\chi(\Omega)\]
where the sum is due on all the critical points $z_j$ of $u_n$ in $\Omega$.  Next we observe that the points $x_{i,n}$ are the unique critical points of $u_n$ in $B_\rho(P_i)$ for every $i=1,..,m$ by Lemma \ref{lem1} and they are nondegenerate maxima so that $index_{x_{i,n}}(\nabla u_n)=1$. This gives the claim.
\end{proof}
\begin{corollary}\label{prop-uguali}
Let $u_n$ be a sequence of solutions to \eqref{1} which blows-up at $\{P_1,\dots,P_m\}$ as $n\to \infty$.  Then the critical points of $u_n$ in $D$  converge to the critical points of $K(x)$ (see \eqref{kx}) and, for $n$ large,  it holds
\begin{equation}\label{3.23}
\begin{split}
\chi(\Omega)-m&=\sum_{z_j\in \mathcal C_n} index_{z_j} (\nabla u_n)
=\sum_{z_j\in \mathcal C} index_{z_j} (\nabla K)
\end{split}
\end{equation}
where $\mathcal C_n$ and $\mathcal C$ are the sets of critical points of $u_n$ and $K$ in the set $D$ respectively.
\end{corollary}
\begin{proof}
It is a consequence of the convergence  \eqref{3} together with  \eqref{prima}.
\end{proof}
Now we are in position to start the proof of Theorem \ref{prop-general}. Since the construction of the domain $\tilde\Omega$ is quite lengthy, we will divide the proof into two parts. In the first one we will prove the formulas \eqref{nc} and \eqref{nb} and subsequently we will construct  the domain $\tilde\Omega$ and the corresponding family of solutions $ u_\l$.

\begin{proof}[Proof of Theorem \ref{prop-general}]
{\bf Step 1: proof of \eqref{nb} and  \eqref{nc}}\\
It is sufficient to prove the result for any sequence of values $\lambda_n\to 0$.
Since $\Omega$ has $k$ holes then $\chi(\Omega)=1-k\leq 0$ and \eqref{3.23} gives that 
\[\sum_{z_j\in \mathcal C_n} index_{z_j} (\nabla u_n)=1-k-m\leq -1\]
so that the solutions $u_n$ have at least one nondegenerate saddle point in $D$.  In the general case, since $index_{z_j}(\nabla u_n)\in \{1,0,-1\}$ we can only say that $u_n$ admits at least $m+k-1$
 nondegenerate saddle points in $D$ and, recalling that we have $1$ maximum in each $B_\rho(P_j)$, $j=1,..,m$, we get the existence of at least $2m+k-1$ critical points in $\Omega$ which proves \eqref{nc}.
\end{proof}
Now we start the construction of the domain $\tilde\Omega$. Let us introduce some notations.
\vskip0.2cm
{\bf The dumbell domain}
\vskip0.2cm
Let $\Omega_0:=\Omega_1\cup\dots,\cup\ \Omega_m$,  where $\Omega_1,\dots, \Omega_m$ are are $m$ smooth, simply connected, bounded domains in $\R^2$ such that $\Omega_i\cap \Omega_j=\emptyset$ if 
$i\neq j$.  Assume that $\Omega_i\subset\{(x,y)\in \R^2: a_i\leq x\leq b_i \},$  for some $b_i<a_{i+1}$ and $\Omega_i\cap \{y=0\}\neq 0$ for $i=1,\dots,m$.  Let $V_\varepsilon:=\{(x,y)\in \R^2: |y|\leq \varepsilon, x\in (a_1,b_m)\}$.  Let $\Omega_\varepsilon$ be a smooth simply connected domain such that 
$\Omega_0\subset \Omega_\varepsilon\subset \Omega_0\cup V_\varepsilon$.
We say that $\Omega_\varepsilon$ is a $m-dumbell.$  \\
In particular each $\Omega_j$ can be a suitable ball of the same radius centered on the $x$-axis.
\begin{tikzpicture}
  \draw[black] (0,0) circle (1cm); 
  \draw[white, line width=3cm] (2, -0.15) arc (0:16:2cm);
  \draw[black]  (4,0) circle (1cm); 
  \draw[white, line width=3cm] (3, -0.15) arc (-4:1:3.1cm);
  \draw[black] (8,0) circle (1cm); 
  \draw[white, line width=3cm] (5, -0.15) arc (0:8:2cm);
   \draw[white, line width=3cm](7, -0.15) arc (0:8:2cm);
  \draw[black]  (1,-0.15) -- (3,-0.15); 
  \draw[black]  (1,0.15) -- (3,0.15); 
  \draw[black]  (5,-0.15) -- (7,-0.15); 
  \draw[black]  (5,0.15) -- (7,0.15); 
\end{tikzpicture}
\vskip0.2cm
\centerline{Example of dumbell with $m=3$}
\vskip0.2cm
Next Lemma was basically proved in  \cite{egp}. We repeat it for reader's convenience.
\begin{lemma}\label{prop-ex1}
For any $m\geq 2$ there exists an $m$-dumbell $\Omega_\varepsilon$ 
such that problem \eqref{1} has one family of solutions $u_\l$, for $\lambda$ small enough, which blow up at $m$ points $\{P_1,\dots,P_m\}$ in $\Omega_\e$ as $\lambda\to 0$. 
\end{lemma}
\begin{proof}

The proof follows as in Theorem 5.5 of \cite{egp} observing that the $Kirkhhoff-Routh$  for the limit domain $\Omega_0$ has a strict local maximum.  
From this they deduce that when $\e$ is small enough, the $Kirkhhoff-Routh$ function $\mathcal {KR}_{\Omega_\e}(x_1,\dots,x_m)$ also has a strict local maximum in $(P_1,\dots,P_m)$ which is stable.  By  \cite{egp} it generates a family of solutions $u_\l$ to \eqref{1} that blow-up at  $\{P_1,\dots,P_m\}$ as $\l\to 0$.
\end{proof}
\begin{remark}\label{rem-convesso}
We can choose the domains $\Omega_j$ that are convex for every $j=1,\dots,m$. In this case by a result of Caffarelli and Friedman in \cite{cf}, the unique critical point of the Robin Function in $\Omega_j$ is nondegenerate. This gives the nondegeneracy of the critical point of the $Kirchhoff-Routh$ function of $\Omega_\e$ at $\{P_1,\dots,P_m\}$. In this case the existence of a family of  blowing up 
solutions $u_\l$ follows by \cite{bapa}.
\end{remark}
\begin{remark}\label{remeps}
The smallness of $\e$ is only need to have the existence of the solution $u$ to \eqref{1}.  
From now we fix such a $\e$ and denote $\Omega_\varepsilon$ by $\Om$.
\end{remark}

We are in position to construct the domain $\tilde\Om$ and the family of
solutions $ u_\l$ to \eqref{1} when $k=0$. The final example will be obtained modifying it appropriately.
\begin{proposition}\label{dumb}
If the dumbell 
is symmetric with respect to the $x$-axis then, for $\lambda$ small enough, there exists a family of solutions $u_\l$, which blow up at $m$ points $P_1,\dots,P_m$ as $\l\to 0$, symmetric with respect to the $x$-axis, with exactly $m-1$ nondegenerate saddle points and $m$ maxima in $\Om$.
\end{proposition}

\begin{tikzpicture}
  \draw[blue] (0,0) circle (1cm); 
  \draw[white, line width=3cm] (2, -0.15) arc (0:16:2cm);
  \draw[blue]  (4,0) circle (1cm); 
  \draw[white, line width=3cm] (3, -0.15) arc (-4:1:3.1cm);
  \draw[blue]  (8,0) circle (1cm); 
  \draw[white, line width=3cm] (5, -0.15) arc (0:8:2cm);
   \draw[white, line width=3cm](7, -0.15) arc (0:8:2cm);
  \draw[blue]  (1,-0.15) -- (3,-0.15); 
  \draw[blue]  (1,0.15) -- (3,0.15); 
  \draw[blue]  (5,-0.15) -- (7,-0.15); 
  \draw[blue]  (5,0.15) -- (7,0.15); 
   \draw[->] (-2.5,0) -- (10,0) node[right] {$x$};
   \draw[->] (4,-2) -- (4,2) node[right] {$y$};
   \fill (0,0) circle (1pt) node[anchor=north] {$x_{1,n}$};
     \fill (4,0) circle (1pt) node[anchor=north] {$x_{2,n}$};
       \fill (8,0) circle (1pt) node[anchor=north] {$x_{3,n}$};
            \draw[red, fill=red] (2,0) circle (1pt) node[anchor=north] {$z_{1,n}$};
                 \draw[red, fill=red] (6,0) circle (1pt) node[anchor=north] {$z_{2,n}$};
\end{tikzpicture}
\vskip0.2cm
\centerline{Example of dumbell with $m=3$}

\begin{proof}
If $\Omega$ is symmetric with respect to the $x$-axis we can construct a symmetric family of solutions $u_\l$ working in the space of functions that are even with respect to the $x$-axis and following \cite{egp}.\\
By Proposition \ref{lem1} and the evenness of $u_\l$ we have that the strict maximum points $x_{1,\l},\dots, x_{m,\l}$ are on the $x$-axis. Next, using Rolle's Theorem
we get the existence of points $(z_{1,\l},0),\dots, (z_{m-1,\l},0)$ such that $\frac{\partial u_\l}{\partial x}(z_{i,\l})=0$ for $i=1,..,m-i$. Finally since $\frac{\partial u_\l}{\partial y}\big|_{y=0}=0$ we get that $z_{1,\l},\dots, z_{m-1,\l}$ are saddle points for $u_\l$ (where we are considering only the first coordinate of the points since the other is always zero). Without loss of generality assume that $x_{1,\l}<z_{1,\l}<x_{2,\l}<z_{2,\l}<\dots<x_{m,\l}$.\\
Next aim is to show that {\bf no} other critical points occur for $u_\l$ when $\lambda$ is small enough.\\
Let us consider a sequence of values $\lambda_n\to 0$.
By Proposition \ref{lem1} and \eqref{3}, up to a subsequence, $z_{i,\l_n}:= z_{i,n}$ converge to a critical points $z_i$ of $K(x)$ in \eqref{kx} verifying $P_1<z_1<P_2<z_2<\dots<P_m$. Then the limit function $K(x)$ has at least $m-1$ critical points $\{z_1,\dots,z_{m-1}\}$ in $\Omega\setminus\{P_1,\dots,P_m\}$.\\
By \eqref{3.23} we know that 
\[\sum_{z_j\in \mathcal C} index_{z_j} (\nabla K)=1-m
\]
if, as before,  $\mathcal C$ denotes the set of critical points of $K(x)$ in $\Om\setminus\{P_1,\dots,P_m\}$. Since, by Proposition \ref{prop-critical-points-K} $index_{z_j}(\nabla K)\leq -1$ for every $z_j\in \mathcal C$ and we have at least $m-1$ critical points, it can only happen that $\mathcal C=\{z_1,\dots,z_{m-1}\}$ and $ index_{z_j}(\nabla K)= -1$ for every $j=1,\dots,m-1$. Then each $z_j$ is a nondegenerate saddle point for $K(x)$.  This implies in turn that also $u_{\l_n}:= u_n$ cannot have other critical points in $\Om$. Since the result holds for any sequence $\lambda_n\to 0$ then it holds for the family $u_\l$.
\end{proof}
Now we are in position to complete the proof of Theorem \ref{prop-general}.
\begin{proof}[Proof of existence of $\tilde\Om$ and $u_\l$ of Theorem \ref{prop-general}]
We start considering the case $m=1$.  Take a bounded smooth domain $\Omega_1$ which is symmetric with respect to the $x$-axis. (We can consider one of the previous $\Omega_i$). It is well known that, since  the Robin function $\mathcal R_{\Omega_1}$ in $\Omega_1$ has a local minimum point $P_1$, there is a family of solutions to \eqref{1} in $\Omega_1$ that blow-up at a point $P_1$.

Next we add a handle $\mathcal{C}_1$ which connects two symmetric points with respect to the $x$ axis,  (see figure below).  We call it a ``lateral handle'' and the corresponding domain has $one$ hole (namely $k=1$).

\begin{tikzpicture}

\begin{scope}
\clip [] (0,0) circle [radius=2cm-0.5\pgflinewidth]; 
\fill[white]  (0,0) circle [radius=2cm];
\end{scope}
 \draw [red,thick,domain=68:292] plot ({-1.8+cos(\x)}, {sin(\x)});
   \draw [red,thick,domain=79:281] plot ({-1.8+0.8*cos(\x)}, {0.8*sin(\x)});
  \draw [blue,thick,domain=0:115] plot ({-1+cos(\x)}, {sin(\x)});
   \draw [blue,thick,domain=245:360] plot ({-1+cos(\x)}, {sin(\x)});
    \draw [blue,thick,domain=130:230] plot ({-1+cos(\x)}, {sin(\x)});
       
      \draw[->] (-4.5,-0.15) -- (1,-0.15) node[right] {$x$}; 
   \draw[->] (-1,-2) -- (-1,2) node[right] {$y$};
   \fill (-1,-0.15) circle (1pt) node[anchor=north] {$x_{1,n}$};
             
                    \fill (-3,-0.15)  node[anchor=north] {\large$\mathcal{C}_1$};
 \end{tikzpicture} 
 
As in Proposition \ref{dumb}, choosing the lateral handle sufficiently thin, we get the existence of a family of blowing-up symmetric solutions for the new domain $\tilde\Om$. 
Next the additional ``lateral handle'' provides the existence of an additional saddle point $z_{1,\l}$; this can be easily seen observing that $\frac{\partial u_\l}{\partial y}=0$ on $\mathcal{C}_1\cap\{x=0\}$ and  $u_\l=0$ on $\partial\mathcal{C}_1\cap\{x=0\}$. Hence again Rolle's Theorem provides the existence of a critical point $z_{1,\l}$ to $u_\l$ (for every $\l$ small enough).\\
Then, we get the existence of exactly $2$ critical points, all nondegenerate.\\
Iterating this procedure adding $\mathcal{C}_1,\mathcal{C}_2,..,\mathcal{C}_k$ ``lateral handles'' which are symmetric with respect to the $x$ axis (see figure 1 in the Intoduction) we construct a domain $\tilde\Om$ with $k$ holes that has a family of $1$-point blowing-up solutions $u_\l$
that have exactly $k+1$ critical points.\\
Next, we turn to the case $m\geq 2$. 
Let us consider the same symmetric dumbell as in Proposition \ref{dumb} and add a handle $\mathcal{C}_1$ which connects two symmetric points with respect to the $x$ axis, both belonging to the first component $\Omega_1$ (see figure below). 
\begin{tikzpicture}
\draw[red] (-1.97,-0.03) arc
    [start angle=0,
        end angle=350,
        x radius=0.4cm,
        y radius =0.4cm] ;

\draw[red] (-1.65,-0.15) arc
    [start angle=0,
        end angle=315,
        x radius=0.7cm,
        y radius =0.6cm] ;
\begin{scope}
\clip [] (0,0) circle [radius=1.9cm-0.5\pgflinewidth]; 
\fill[white]  (0,0) circle [radius=2cm];
\end{scope}
 \draw[blue] (0,0.05) arc
    [start angle=10,
        end angle=348,
        x radius=1cm,
        y radius =1cm] ;
        
      \draw[blue] (4,0.05) arc
    [start angle=9,
        end angle=171,
        x radius=1cm,
        y radius =1cm] ;  
        
         \draw[blue] (4,-0.35) arc
    [start angle=-13,
        end angle=-168,
        x radius=1cm,
        y radius =1cm] ;    
        
        \draw[blue] (6,0.05) arc
    [start angle=170,
        end angle=-168,
        x radius=1cm,
        y radius =1cm] ;
        
  \draw[blue]  (0,0.042) -- (2.02,0.042); 
    \draw[blue]  (0,-0.342) -- (2.02,-0.342); 
    \draw[blue]  (4,0.042) -- (6.02,0.042); 
    \draw[blue]  (4,-0.342) -- (6.02,-0.342); 
      \draw[->] (-4.5,-0.15) -- (9,-0.15) node[right] {$x$}; 
   \draw[->] (3,-2) -- (3,2) node[right] {$y$};
   \fill (-1,-0.15) circle (1pt) node[anchor=north] {$x_{1,n}$};
     \fill (3,-0.15) circle (1pt) node[anchor=north] {$x_{2,n}$};
       \fill (7,-0.15) circle (1pt) node[anchor=north] {$x_{3,n}$};
            \draw[red, fill=red] (1,-0.15) circle (1pt) node[anchor=north] {$z_{1,n}$};
                 \draw[red, fill=red] (5,-0.15) circle (1pt) node[anchor=north] {$z_{2,n}$};
                 
                    \fill (-3,-0.15)  node[anchor=north] {\large$\mathcal{C}_1$};
 \end{tikzpicture}  

As in Proposition \ref{dumb}, choosing the lateral handle sufficiently thin, we get the existence of a sequence of blowing-up symmetric solutions for the new domain $\tilde\Om$. Of course we again have the existence of at least $m$ strict maximum points $x_{1,\l},\dots, x_{m,\l}$ and $m-1$  saddle points $z_{1,\l},\dots, z_{m-1,\l}$ as in Proposition \ref{dumb} and, as in the case $m=1$, each ``lateral handle'' provides  an additional saddle point $z_{m,\l}$  to $u_\l$
As in Proposition \ref{dumb}  we get the existence of exactly $2m$ nondegenerate critical points.\\
Iterating this procedure adding $\mathcal{C}_1,\mathcal{C}_2,..,\mathcal{C}_k$ ``lateral handles'' which are symmetric with respect to the $x$ axis we construct a domain $\tilde\Om$ with $k$ holes that has a family of blowing-up solutions $u_\l$ 
that have exactly $2m+k-1$ critical points.
\end{proof}

\section{Proof of Theorems \ref{T1} and \ref{ex}}\label{S5}
\begin{proof}[Proof of Theorem \ref{T1}]
Let $\lambda_n$ be a sequence of values such that $\lambda_n\to 0$, and $u_n$ be the corresponding solution.
By Proposition \ref{prop-preuguali} we have, for $n$ large enough, that 
\[\sum_{z_j\in \mathcal C_n} index_{z_j} (\nabla u_n)=\chi(\Omega)-m
\]
where $\mathcal C_n$ is the sets of critical points of $u_n$ in $D$. Let us consider the following cases,
\vskip0.2cm
\begin{itemize}
\item If $m=1$ and $\Om$ is simply connected ($\chi(\Omega)=1$) then 
$$\sum_{z_j\in \mathcal C_n} index_{z_j} (\nabla u_n)=0$$
 and, by Corollary \ref{prop-uguali} this implies that $\sum_{z_j\in \mathcal C} index_{z_j} (\nabla K)=0$, which, together with the properties of the critical points of $K(x)$ in $D$ in Proposition \ref{prop2} implies that $K(x)$ has no critical points in $D$. The $C^1$ convergence of $u_n$ to $K$ in $D$ implies that also $u_n$ has no critical points in $D$ for $n$ large enough. Then the unique critical point is the maximum  $x_{1,n}$ which is nondegenerate by Lemma \ref{lem1}.
\vskip0.2cm
\item If $m=2$ and $\Om$ is simply connected we have that 
$$\sum_{z_j\in \mathcal C_n} index_{z_j} (\nabla u_n)=-1.$$
 Then Proposition \ref{prop-critical-points-K} implies that $K(x)$ has a unique nondegenerate critical point $x_0$ and the $C^2$ convergence of $u_n$ to $K$ implies that $u_n$ has a nondegenerate critical point $x_{0,n}\to x_0$. This gives the uniqueness and nondegeneracy of the critical point of $u_n$.
 \item If $m=1$ and $\Om$ has one hole we have that $\chi(\Omega)=0$ and so
 $$\sum_{z_j\in \mathcal C_n} index_{z_j} (\nabla u_n)=-1.$$
 Arguing as in the previous step we deduce the existence of $exactly$ one nondegenerate critical point in $D$. This fact, jointly with the uniqueness of the maximum point in the ball $B_\rho(P_1)$, gives the claim.
\end{itemize}
Since the result holds for every sequence $\lambda_n$ we get the claim for any $\l$ small enough.
\end{proof}
Now we give the proof of Theorem \ref{ex}. As in the Theorem \ref{prop-general}, the construction of the domain needs some  definitions and lemmas.

As in the previous section we start by considering a $contractible$ domains.
\vskip0.2cm
{\bf Case 1: contractible domains}
\vskip0.2cm
Let us consider a regular polygon with $m$ sides of length $1$ and barycenter at the origin $O$.  At each vertex $Q_i$ of the polygon, $i=1,..,m$ we place a ball $B_i$ of radius $r<\frac{1}{4}$ centered at $Q_i$. Then we let $\Omega_0=\cup_{i=1}^mB_i$ and, 
by Remark \ref{rem-convesso}, the point $(Q_1,..,Q_m)$ is nondegenerate for the $Kirkhhoff-Routh$ function in $\Omega_0$.  
Now, we connect each component $B_i$ of $\Omega_0$ with the barycenter by $m$ straight thin tubes of thickness $\epsilon$
(see figure below for $m=3$). Finally we smooth the corners at the  boundary of $B_i$ to obtain a smooth set $\Omega_\epsilon$.\\
 Alternatively, instead of balls  we can consider $m$ copies of a convex domain. 
 \begin{center}
\begin{tikzpicture}[scale=0.5]

\draw[blue] (0,0) arc
    [start angle=-20,
        end angle=320,
        x radius=1cm,
        y radius =1cm] ;
        
            \draw[blue] (6,0) arc
    [start angle=200,
        end angle=-140,
        x radius=1cm,
        y radius =1cm] ;

        

        
           \draw[blue] (3,-5) arc
    [start angle=100,
        end angle=440,
        x radius=1cm,
        y radius =1cm] ;

 \draw[blue]  (0,0) -- (3.1,-1.7); 
  \draw[blue]  (-0.16,-0.3) -- (2.95,-2.1); 
    \draw[blue]  (6,-0.025) -- (3.1,-1.7); 
   \draw[blue]  (6.2,-0.3) --(3.3,-2.1) ;  
      \draw[blue]  (3.3,-2.1) -- (3.35,-5);  
      \draw[blue]  (2.95,-2.1) -- (3,-5);  

   \fill (-1,0.8)  node[anchor=north] {$B_1$};
     \fill (7,0.8)  node[anchor=north] {$B_2$};
      \fill (3.2,-5.5)  node[anchor=north] {$B_3$};
      \fill (8,-2.8)  node[anchor=north] {The domain $\Om_\e$ for $m=3$};

 \end{tikzpicture}  
\end{center}

The domain $ \Omega_\e$ is contractible for every $\e$ and it is invariant by the action of the 
the group of rotations that fix the barycenter and rotate by an angle of $\theta=\frac {2\pi}m$.
As  in the previous example then we have
\begin{lemma}\label{lemma-6.1}
For $\e$ small enough, problem \eqref{1} has in $\Omega_\varepsilon$ at least one family 
of solutions $u_\l$ which blow up at the points $\{P_1,..,P_m\}$ as $\l\to 0$. These solutions $u_\l$ are invariant by a rotation of angle $\theta=\frac {2\pi}m$.
\end{lemma}
\begin{proof}
The proof follows again as in Theorem 5.5 of \cite{egp}. Here we set our problem in the space of functions invariant by a rotation of angle $\theta=\frac {2\pi}m$ and observe that the $Kirkhhoff-Routh$ function $\mathcal {KR}_{\Omega_\e}$, which is invariant by a rotation of angle $\theta=\frac {2\pi}m$,  has a nondegenerate critical point in
in $(P_1,..,P_m)$ (which is near $(Q_1,\dots,Q_m)$). 
\end{proof}
\begin{remark}
As in Remark \ref{remeps} we fix $\e$ small and set $\Om=\Omega_\e$.
\end{remark}

In next proposition we compute the number of critical points.
\begin{proposition}\label{prop1}
For every $m\geq 3$ there exists a domain $\Om$ such that \eqref{1} has at least one family of solutions $u_\l$ which blow up at $m$ points $\{P_1,\dots,P_m\}$ in $\Omega$ as $\lambda\to 0$.   The solutions $u_\l$ are invariant by a rotation of angle $\theta=\frac {2\pi}m$ and have at least $2m+1$ nondegenerate critical points whose  $m$ are saddle points and $m+1$ maxima (one of the maxima coincides with the barycenter $O$). Moreover as $\l\to 0$ the saddle points converge to the the barycenter $O$ which becomes a degenerate saddle of index $m-1$ for the function $K(x)$. 
\end{proposition}
\begin{center}
\begin{tikzpicture}[scale=0.5]

\draw[blue] (0,0) arc
    [start angle=-20,
        end angle=320,
        x radius=1cm,
        y radius =1cm] ;
        
            \draw[blue] (6,0) arc
    [start angle=200,
        end angle=-140,
        x radius=1cm,
        y radius =1cm] ;

           \draw[blue] (3,-5) arc
    [start angle=100,
        end angle=440,
        x radius=1cm,
        y radius =1cm] ;

 \draw[blue]  (0,0) -- (3.1,-1.7); 
  \draw[blue]  (-0.16,-0.3) -- (2.95,-2.1); 
    \draw[blue]  (6,-0.025) -- (3.1,-1.7); 
   \draw[blue]  (6.2,-0.3) --(3.3,-2.1) ;  
      \draw[blue]  (3.3,-2.1) -- (3.35,-5);  
      \draw[blue]  (2.95,-2.1) -- (3,-5);  

   \fill (-1,0.5) circle (1pt) node[anchor=north] {$x_{1,\l}$};
     \fill (7,0.5) circle (1pt) node[anchor=north] {$x_{2,\l}$};
      \fill (3.1,-1.9) circle (2pt) node[anchor=north] {$x_{4,\l}\equiv O$};
      \fill (3.2,-5.8) circle (1pt) node[anchor=north] {$x_{3,\l}$};
           \draw[red, fill=red] (1.3,-0.95) circle (1pt) node[anchor=north] {$z_{1,\l}$};
                 \draw[red, fill=red] (4.8,-0.95) circle (1pt) node[anchor=north] {$z_{2,\l}$};
                  \draw[red, fill=red] (3.15,-3.95) circle (1pt) node[anchor=north] {$z_{3,\l}$};
       \fill (12,-2.8)  node[anchor=north] {Critical points of $u_\l$ in $\Om$ for $m=3$};                
 \end{tikzpicture}  
\end{center}
\begin{proof}
The existence of the solutions $u_\l$ is given by Lemma \ref{lemma-6.1}.  We only have to compute the number of the critical points along a sequence $\lambda_n\to 0$.
By formula \eqref{nb} we get
\begin{equation}\label{fo}
\sum_{z_j\in \mathcal C_n} index_{z_j} (\nabla u_n)=1-m
\end{equation}
where we recall that $\mathcal C_n$ is the set of critical points of $u_n$ in $\Om\setminus\cup_{i=1}^m B_{\rho}(P_i)$, for $\rho$ small enough. Using \eqref{ind} we deduce that $u_n$ has at least $m-1$ saddle points in  $\Om\setminus\cup_{i=1}^m B_{\rho}(P_i)$.  Let us show that $u_n$ has at least a $m-th$ saddle points. Indeed, using the simmetry of $u_n$, if no other saddle point occurs, then the $m-1$ saddle points must coincide with the barycenter $O$. This is not possible because we shall have a critical point of index $m-1\ge2$, a contradiction with \eqref{ind}.

Hence $u_n$ admits at least $m$ saddle points of index $-1$.  But if no other critical point occurs, we have that $\{z_{1,n},..,z_{m,n}\}= \mathcal C_n$ and $\sum index_{z_{j,n}} (\nabla u_n)=-m$ a contradiction with \eqref{fo}. 

Therefore, there must be at least one additional critical point of index $1$ (a maximum)  to $u_n$ which is necessarily the origin. Otherwise, by symmetry reasons, there would be other $m$ maximum points and the total degree of $\nabla u_n$ should be $0$, again a contradiction. So the maximum point is located at the origin.
 
This proves the claim on the number of critical points. Observe that (although it seems unlikely) we cannot exclude the existence of other $m$ maxima and $m$ saddles of $u_n$.

We end the proof discussing  the behavior of the saddle points $\{z_{1,n},..,z_{m,n}\}$ when $n\to \infty,$ ( $\lambda_n \to0$). Passing to the limit we get
\[\sum_{z_j\in \mathcal C} index_{z_j} (\nabla K)=1-m
\]
where $\mathcal C$ is the set of critical points of $K(x)$ in $\Om\setminus\cup_{i=1}^m B_{\rho}(P_i)$.
However, since $K(x)$ is a harmonic function, the origin cannot be a point of maximum, but rather must be a saddle. This means that as $\l_n \to0$, the $m$ saddles $z_{j,n}$ must collapse to the barycenter, i.e. for any $j=1,..,m$, $z_{j,n} \to O$ as $n \to\infty$, and the point $O$ becomes a degenerate saddle point of index $1-m$ for $K(x)$. \\
Since $K(x)$ admits a unique critical point which is the barycenter, then every critical point of $u_n$ in $\mathcal C_n$ must collapse there.
\end{proof}
{\bf Case 2: domains with holes, a special case}
\vskip0.2cm
Before proving the general result for a domain with $k$ holes, let us consider the special case $k=hm$ for some positive integer $h$. Although this is a particular case, there is the advantage that it is much simpler.

Indeed this case can be proved straightforwardly, adding $m$ ``lateral handles'' to the domain constructed in Proposition \ref{prop1},  and reasoning as in the proof of Theorem \ref{ex}.
\begin{proposition}\label{prop-manici-tanti-critici-1}
For every $m\geq 3$
there exists a domain $\Omega$ with $hm$ holes such problem \eqref{1} has at least one family of solutions $u_\l$ which blow up at $m\geq 3$ points $\{P_1,\dots,P_m\}$ in $\Omega$ as $\l\to 0$. The solutions $u_\l$ are invariant by a rotation of angle $\theta=\frac {2\pi}m$ and have at least $(h+2)m+1 $ critical points whose  $(h+1)m$ are nondegenerate saddle points and $m+1$ maxima (one of the maxima coincides with the barycenter $O$).
 Moreover as $\l \to0$, $m$ saddle points converge to the the barycenter $O$ which becomes a degenerate saddle point of index $m-1$ for the function $K(x)$. 
\end{proposition} 
\begin{proof}
In proposition \ref{prop1} we considered solutions that are invariant under rotations by an angle $\theta=\frac {2\pi}m$. The same procedure can be applied by additionally requiring that the solutions are also invariant under reflection with respect to the line passing through the center of one of the ball $B_i$ and the barycenter $O$, since the $Kirkhhoff-Routh$ function $\mathcal {KR}_{\Omega}$ has also this symmetry.\\
Next, as in the proof of Theorem \ref{ex}, we add one or more handles to the balls that make up the domain $\Om$, as we did in the case of the dumbbell. Due to the invariance by rotation each handle creates $m$ holes, and $h$ handles produce $hm$ holes.  Similarly to the dumbbell case, for the invariance under reflection, each handle adds a critical saddle point and the claim follows.
\end{proof}
\vskip0.2cm
{\bf Case 3: domains with holes, the general case}
\vskip0.2cm
Here, for every $\lambda$ small fixed, we construct a domain $\widehat \Omega$ which has $k\geq 1$ holes and such that the solution $\widehat u_\l$ has at least $2m+k+1$ nondegenerate critical points.\\
The construction is a little bit delicate because involves different parameter $R,\l,\d$ which must be fixed independently.
\vskip0.2cm
{\bf The domain $\Om_{\d}$}. 
Let us consider the contractible domain $\Omega$ of Proposition \ref{prop1} and let $ u_\l$ be the corresponding family of solutions.  We choose a point $x_0\in \Omega$ such that $x_0\neq \{O, P_1,\dots,P_m\}$. Then $x_0$ is not a critical point of $u_\l$ if $\l$ is small enough.\\
We consider a small ball $B_R(x_0)\subset \Omega\setminus\{O,P_1,\dots,P_m\}$ where $R$ is such that 
\[\lambda^*(\Omega)\pi R^2+R\int_{\partial B_R(x_0)}\left( 12\left|\nabla 8\pi \sum_{i=1}^m G(x,P_i)\right|^2 +\lambda^*(\Omega)\left( e^{4 \sum_{i=1}^m 8\pi G(x,P_i)}+1\right)\right)
d\sigma<\frac {\e_0}4\]
where $\e_0$ is as defined in \eqref{epsilon-zero} and $\lambda^*(\Omega)$ is the maximum value such that problem \eqref{1} has solutions.\\
This fix the value of $R$.

Now we choose $\lambda$ small such that, on $\partial B_R(x_0)$,
\[ u_\l(x)\leq 2 \sum_{i=1}^m 8\pi G(x,P_i)\]
and 
\[|\nabla u_\l(x)|\leq 2\left|\nabla  \sum_{i=1}^m 8\pi G(x,P_i)\right|\]
(see \eqref{3}.) Choosing if necessary $\lambda$ smaller, we can assume
that the solutions $u_\l$ in Proposition \ref{prop1} are nondegenerate and 
have $2m+1$ nondegenerate critical point in $\Omega$ (see Theorem \ref{T5}).\\
This fix the value of  $\lambda$. \\
Next we remove a small ball $B_\d(x_0)\subset B_R(x_0)\subset \Omega$ (see figure below). We call this new domain $\Om_{\d}$.
\begin{tikzpicture}[scale=0.5]
\draw[red] (-1,-0.3) ellipse (0.2);

\draw[blue] (0,0) arc
    [start angle=-20,
        end angle=320,
        x radius=1cm,
        y radius =1cm] ;
        
            \draw[blue] (6,0) arc
    [start angle=200,
        end angle=-140,
        x radius=1cm,
        y radius =1cm] ;
            \draw[blue] (3,-5) arc
    [start angle=100,
        end angle=440,
        x radius=1cm,
        y radius =1cm] ;

 \draw[blue]  (0,0) -- (3.1,-1.7); 
  \draw[blue]  (-0.16,-0.3) -- (2.95,-2.1); 
    \draw[blue]  (6,-0.025) -- (3.1,-1.7); 
   \draw[blue]  (6.2,-0.3) --(3.3,-2.1) ;  
      \draw[blue]  (3.3,-2.1) -- (3.35,-5);  
      \draw[blue]  (2.95,-2.1) -- (3,-5);  
  
   \fill (-1,0.8)  node[anchor=north] {$B_1$};
     \fill (7,0.8)  node[anchor=north] {$B_2$};
      \fill (3.2,-5.5)  node[anchor=north] {$B_3$};
        \fill (-1,-0.7)  node[anchor=north] {\color{red}$B_\d(x_0)$};
      \fill (8,-2.8)  node[anchor=north] {The domain $\Om_{\d}$};
 \end{tikzpicture} 
\vskip0.2cm 

The construction of the solutions $u_\d$ in $\Omega_\d$ will be given in the following steps:
\vskip0.2cm
Step 1. {\em Existence and nondegeneracy  of the solution $u_\d$ in $\Om_{\d}$.}
\vskip0.2cm
Step 2. {\em $||u_\d||_{L^\infty(\Omega)}\le C_0$ with $C$ independent of $\d$ and $u_\d\to u_\l$ uniformly outside of compact sets containing $B_\d(x_0)$.}
\vskip0.2cm
Step 3. {\em Uniqueness of the critical point of $u_{\d}$ near the hole.}
\vskip0.2cm
{\bf Proof of Step 1.}  
By Dancer's results in \cite{da1,da2} and the remarks therein, using the nondegeneracy of $u_\l$ we  get that, for $\d$ small enough, there exists a family of solutions $u_\d$ in $\Om_{\d}$ such that $u_\d\to u_\l$ in $L^p(\Om)$ for every $p>1$. Moreover, again by Theorem $1$ in \cite{da1}, the nondegeneracy of $ u_\l$ implies that of $u_\d$, again for $\d$ small.

\vskip0.2cm
{\bf Proof of Step 2.}  
By the standard regularity theory we have that $u_\d\to u_\l$ in any compact set outside of $B_\d(x_0)$. Hence it is enough to prove our claim in $B_R(x_0)$. Again using the standard regularity theory we get
 $u_\d\to u_\l$ in $C^1(\partial B_R(x_0))$. In particular $u_\d\le C_R$ is uniformly bounded on $\partial B_R(x_0)$, where $C_R$ is a positive constant depending only on $R$.
 
We can choose $\delta$ so small that, on $\partial B_R(x_0)$,
\[u_\delta (x)<2 u_\l(x)<4 \sum_{i=1}^m 8\pi G(x,P_i)\]
and
\[\left| \nabla u_\delta\right|\leq 2 \left| \nabla u_\lambda\right|<4 \left|\nabla  \sum_{i=1}^m 8\pi G(x,P_i)\right|\]
by the previous assumptions on $\lambda$. Now we apply the Pohozaev identity to $u_\d$ in $B_R(x_0)\setminus B_\d(x_0)$, (see \cite{po}) with $F(u)=\lambda\left(e^u-1\right)$. 
Then we get
\[
\begin{split}
&2\lambda \int _{B_R(x_0)\setminus B_\d(x_0)}\left( e^{u(x)}-1\right) dx=\\
&\int_{\partial B_R(x_0)} \left(\left[(x-x_0)\cdot \nabla u_\d \right] \frac{\partial u_\d}{\partial \nu} -(x-x_0)\cdot \nu \frac {|\nabla u_\d|^2}2+\lambda (x-x_0)\cdot \nu (e^{u_\d}-1)\right)
d\sigma\\
&\int_{\partial B_\d(x_0)} \left(\left[(x-x_0)\cdot \nabla u_\d \right] \frac{\partial u_\d}{\partial \nu} -(x-x_0)\cdot \nu \frac {|\nabla u_\d|^2}2\right)
d\sigma
\end{split}\]
where $\nu$ is the outer normal.  Since $u_\d=0$ on $\partial B_\d(x_0)$ and $B_\d(x_0)$ is starshaped then
\[\int_{\partial B_\d(x_0)} \left(\left[(x-x_0)\cdot \nabla u_\d \right] \frac{\partial u_\d}{\partial \nu} -(x-x_0)\cdot \nu \frac {|\nabla u_\d|^2}2\right)
d\sigma=\frac 12 \int_{\partial B_\d(x_0)} (x-x_0)\cdot \nu |\nabla u_\d|^2 d\sigma\leq 0.\]
Then the previous equality becomes
\[\begin{split}
&2\lambda \int _{B_R(x_0)\setminus B_\d(x_0)}e^{u(x)}dx\leq 2\lambda |B_R(x_0)|+\\
&\int_{\partial B_R(x_0)} \left(\left[(x-x_0)\cdot \nabla u_\d \right] \frac{\partial u_\d}{\partial \nu} -(x-x_0)\cdot \nu \frac {|\nabla u_\d|^2}2+\lambda (x-x_0)\cdot \nu (e^{u_\d}-1)\right)
d\sigma\\
&\leq 2\lambda^*(\Omega) \pi R^2+\int_{\partial B_R(x_0)} \left(\frac 32 |x-x_0||\nabla u_\d|^2+\lambda^*(\Omega)|x-x_0|\left(e^{u_\d}+1\right)\right)d\sigma\\
&\leq 2\lambda^*(\Omega) \pi R^2+ R \int_{\partial B_R(x_0)}\left(24\left| \nabla 8\pi \sum_{i=1}^m G(x,P_i)\right|^2+\lambda^*(\Omega)\left(e^{4\sum_{i=1}^m 8\pi G(x,P_i)}+1\right)\right)d\sigma
\end{split}\]
so that, by the assumption on $R$ we have 
\[\lambda \int _{B_R(x_0)\setminus B_\d(x_0)}e^{u(x)}dx\leq \frac {\e_0} 4.\]
Hence Lemma \ref{lemma-piccolezza} implies that 
\[\| u_\d\|_{L^{\infty}(B_R(x_0))}\leq C\]
where $C$ is independent on $\d$ and on $\lambda$.
\begin{proof}[\bf Proof of Step 3 and Theorem \ref{ex}]
Once we have the existence of a solution $u_{\d}$  to \eqref{1} in $\Om_{\d}$ which is uniformly bounded in $L^\infty(\Omega)$ we can use Theorem \ref{th1-1} with $v_\d=u_{\d}$ and $v_0=u_\l$.  
So the solution  $u_{\delta}$ in $\Om_\d$ has exactly one more saddle point of index $-1$ than $u_\l$ and this gives the claim of Theorem \ref{ex} when $k=1$ and $\l$ is fixed. 

We end the proof showing that the previous procedure can be iterated to handle the case $k\ge1$.

We fix the previous domain $\Om_\d$ with one hole and remove another ball $B_{\d_1}(x_1)\subset\Omega\setminus\{O,P_1,\dots,P_m\}$ and such that $B_{\d_1}(x_1)\cap B_\d(x_0)=\emptyset$. 
Proceding as in the case of $k=1$ we construct
a solution  $u_{\d_1}$ in a domain with $2$ holes which has one more saddle point of index $-1$ with respect to $u_\delta$, and $two$ more saddle points with respect to $\bar u_\l$. The procedure can be iterated removing $k$ balls $B_{\delta_i} (x_i)$ obtaining a solution $u:=u_\l$ in a domain with $k$ holes that has $k$ saddle points more than $\bar u_\l$.

Since this construction can be done for every fixed lambda, we can produce a sequence of domains $\Omega_n$ with $k\ge 1$ holes, and solutions $u_n:=u_{\lambda_n}$ such that $\lambda_n\to 0$ and $u_n$ has at least $2m+k+1$ critical points in $\Omega_n$, $m+1$ of them are local maxima and $m+k$ are saddle points.

Finally, by construction, the sequence $u_n$ blows-up at $\{P_1,\dots, P_m\}$ as $n\to+\infty$.
\end{proof}

\bigskip
\footnotesize
\noindent\textit{Acknowledgments.}
This work has been developed within the framework of the project e.INS- Ecosystem of Innovation for
Next Generation Sardinia (cod. ECS 00000038) funded by the Italian Ministry for Research and Education
(MUR) under the National Recovery and Resilience Plan (NRRP) - MISSION 4 COMPONENT 2, "From
research to business" INVESTMENT 1.5, "Creation and strengthening of Ecosystems of innovation" and
construction of "Territorial R\&D Leaders".
The first author is funded by Fondazione di Sardegna, Uniss,  annual fund installment 2017 and 2020 and by CUP J55F21004240001. The two authors are partially funded by Gruppo Nazionale per l’Analisi Matematica, la Probabilità e le loro Applicazioni (GNAMPA) of the Istituto Nazionale di Alta Matematica (INdAM).

\bibliography{max-fra-gelfand3.bib}
\bibliographystyle{abbrv}

\end{document}